\documentclass{aims}
\usepackage{amsmath}
\usepackage{mathtools}
\usepackage{amssymb}
  \usepackage{paralist}
  \usepackage{tikz}
  \usetikzlibrary{arrows}
\usepackage{ragged2e}
  \usepackage{graphics} 
  \usepackage{epsfig} 

\usepackage{graphicx}  
\usepackage{epstopdf}
\hypersetup{urlcolor=blue, citecolor=red}

\numberwithin{figure}{section}
\numberwithin{equation}{section}
\numberwithin{table}{section}

\def\currentvolume{X}
 \def\currentissue{X}
  \def\currentyear{200X}
   \def\currentmonth{XX}
    \def\ppages{X--XX}
     \def\DOI{10.3934/xx.xx.xx.xx}

\newtheorem{theorem}{Theorem}[section]
\newtheorem{corollary}{Corollary}[section]

\newtheorem{lemma}{Lemma}[section]

\theoremstyle{definition}
\newtheorem{definition}{Definition}[section]
\newtheorem{remark}{Remark}[section]

\newtheorem{assumption}{Assumption}[section]

\title[Synchronization in Networks with Strongly Delayed Couplings]{Synchronization in Networks with Strongly Delayed Couplings}

\author[Daniel M. N. Maia et al.]{}


 \email{maia@physik.hu-berlin.de}
  \email{elbert.macau@inpe.br}
 \email{tiago@icmc.usp.br}
 \email{yanchuk@math.tu-berlin.de}

\begin{document}
\maketitle

\centerline{\scshape Daniel M. N. Maia\footnote{Current address: Department of Mathematics and Statistics, State University of Rio Grande do Norte -- UERN, Mossor\'o-RN, 59610-210 Brazil.}, Elbert E. N. Macau}
\medskip
{\footnotesize
 \centerline{Associate Laboratory of Applied Computing and  Mathematics - LAC}
 \centerline{National Institute for Space Research - INPE}
   \centerline{Av. dos Astronautas, 1758, S\~ao Jos\'e dos Campos-SP, 12227-010, Brazil}
} 

\medskip

\centerline{\scshape Tiago Pereira}
\medskip
{\footnotesize
 \centerline{Institute of Mathematics and Computer Sciences - ICMC}
   \centerline{University of S\~ao Paulo - USP}
   \centerline{Av. do Trabalhador S\~ao-Carlense 400, S\~ao Carlos-SP, 13566-590, Brazil}
}

\medskip

\centerline{\scshape Serhiy Yanchuk}
\medskip
{\footnotesize
 \centerline{Institute of Mathematics, Technical University of Berlin}
   \centerline{Strasse des 17. Juni 136, 10623 Berlin, Germany}
}

\bigskip

 \centerline{(Communicated by the associate editor name)}

\begin{abstract}
We investigate the stability of synchronization in networks of dynamical
systems with strongly delayed connections. We obtain strict conditions
for synchronization of periodic and equilibrium solutions. In particular,
we show the existence of a critical coupling strength $\kappa_{c}$,
depending only on the network structure, isolated dynamics and coupling
function, such that for large delay and coupling strength $\kappa<\kappa_{c}$,
the network possesses stable synchronization. The critical coupling
$\kappa_{c}$ can be chosen independently of the delay for the case
of equilibria, while for the periodic solution, $\kappa_{c}$ depends
essentially on the delay and vanishes as the delay increases. We observe
that, for random networks, the synchronization interval is maximal
when the network is close to the connectivity threshold. We also derive
scaling of the coupling parameter that allows for a synchronization
of large networks for different network topologies. 
\end{abstract}

\section{Introduction}

Coupled systems and networks with time-delayed interactions emerge
in different fields including laser dynamics \cite{Javaloyes2003,Schlesner2003,Fiedler2008,Erneux2009,Soriano2013,YanchukGiacomelli2017},
neural networks \cite{Wu1998,Foss2000,Izhikevich2006,Campbell2006},
traffic systems \cite{Orosz2010}, and others. Time delays in these
systems are caused by finite signal propagation or finite reaction
times. In many cases, especially in optoelectronics, delays are larger
than the other time scales of the system, and they play a major role
for system's dynamics \cite{Soriano2013,YanchukGiacomelli2017}. For neural systems, large delays may emerge
as a lump effect of a signal propagation along feed-forward chains
of neurons \cite{Luecken2013b}.

When the interaction among oscillators is diffusive,
the network possesses a synchronous
subspace, where all subsystems constituting the network behave identically.
The stability of the synchronous subspace plays a major role in applications. 
For instance, neural synchronization is known
to be involved in brain functioning \cite{Singer1993}, synchronization
of lasers is used for communication purposes \cite{Colet1994,Argyris2005}.

Therefore, the stability of synchronization for strongly self-feedback delayed interactions 
has been addressed in various contexts and generalities 
in a mixture of analytical and numerical techniques \cite{Kinzel,Yanchuk2010_2}. 

Here, we aim at studying how  stability of the synchronization manifold depends on the dynamics of the isolated nodes and network structure. We focus on the case where the isolated dynamics is either periodic or an equilibrium solution. This case is amenable to detailed understanding in the limit of large delays.

{\bf Model and discussion of results:}
We consider a network model  with a self-feedback term in the coupling. This kind of coupling is widely studied, for example, in semiconductors lasers models \cite{Soriano2013,Hart2015}, where the self-feedback interaction is generated by the optical feedback.
We consider the following system of $n$ identical oscillators with
time-delayed interactions \cite{Hart2015,Yanchuk2010_2,Kinzel,Steur2014,Dahms2012,Campbell2006}
\begin{equation}
\dot{x}_{j}(t)=f(x_{j}(t))+\kappa\sum_{l=1}^{n}A_{jl}h(x_{l}(t-\tau)-x_{j}(t-\tau)),\label{eq:net_mod_lap}
\end{equation}
where $j=1,\cdots,n$, $\kappa$ is the coupling strength, $A$
is the adjacency matrix with the elements $A_{jl}=1$ if $j\ne l$
and there is a link from $l$ to $j$ and $A_{jl}=0$ otherwise. The function 
$f:\mathbb{R}^{q}\to\mathbb{R}^q$ describes
the uncoupled dynamics and the coupling function $h:\mathbb{R}^{q}\to\mathbb{R}^{q}$, $h(0)=0$,
describes interactions between the nodes.
More details and
assumptions on the model \eqref{eq:net_mod_lap} is given in Sec. \ref{sec:Dynamical-time-delay-network}.

We consider the  time-delay $\tau>0$ as a large parameter.
Physically, this means that the interaction time is much larger than
the typical time scale of an uncoupled system, which often occurs
in, e.g. coupled laser systems \cite{Soriano2013,YanchukGiacomelli2017}.

The uncoupled units of the network have an attractor, 
which is either an equilibrium or periodic orbit. And,
we obtain conditions for synchronization and desynchronization which depends on the network
dynamics.

\begin{itemize}
\item[(\textit{i})] For the equilibrium, we show that with strong enough delay, there
is a critical coupling parameter $\kappa_c$, which depends on the dynamics, coupling function and network structure, such that the equilibrium solution in the network remains stable
if and only if $|\kappa|<\kappa_c$ (see Theorem \ref{theo:fp}). For this case, $\kappa_c$
is independent of the delay.

\item[(\textit{ii})]  Synchronization of stable periodic
orbits  is in sharp contrast to fixed point case. It is shown to be always desynchronized for long enough
time-delay. However, for long but finite delay the synchronization
of periodic orbits are attained for an interval that is shrinking
as the delay grows to infinity, that is, for this case the critical coupling parameter depends
also on the delay, $\kappa_c = \kappa(\tau)$,
and $\kappa(\tau)\to 0$ as $\tau\to\infty$. 
See Theorem \ref{cor_sync_periodc_orbits} for the
precise statement.
\end{itemize}

Fixing the vector field $f$, the coupling function $h$ and the delay $\tau$, we can study how the network structure affects the critical coupling $\kappa_c$. Our results elucidate this dependence.
For instance, we will show that
\begin{itemize}
\item[--] For Barab\'asi-Albert scale free network, $\kappa_c \sim \mathcal{O}(1/\sqrt{n})$.
\item[--] For Erd\"os-R\'eniy (ER) random networks, $\kappa_c \sim \mathcal{O}(1/\log{n})$.
\end{itemize}

This shows that having a large number of connections is detrimental, and
close to percolation threshold is optimal for synchronization. 
This is in contrast to the non-delayed case \cite{Maia2016, Pereira2014} --
best synchronization scenario happens for homogeneous networks and the large the degrees the better.

Here we use the theory of functional differential equations \cite{Hale1993,Smith},
elements of the spectral theory of graphs \cite{Fiedler1973,Mohar2004,Brouwer2012},
as well as a spectral theory for delay differential equations with
strong delays \cite{Yanchuk2010,Lichtner2011,Sieber2012,YanchukGiacomelli2017}.
Some important ideas from Ref.~\cite{Yanchuk2010_2} on the synchronization
of strongly delayed networks are used as well.

The manuscript is organized as follows: In Sec.~\ref{sec:Dynamical-time-delay-network},
we give some assumption on the considered model. In Sec. \ref{sec_main_results} we
state the main results, in particular, a special case for the result of synchronization of periodic orbits. In Section \ref{sec_spectrum_description} we give a description of
 the spectrum for
the variational equation and prove the main results with illustrations given in Sec. \ref{sec_example}. In Sec. \ref{sec_sync_loss} we discuss the relevance of the structure of
the network to the synchronization condition and
 in Sec. \ref{sec_scaling} we discuss how scaling properties
of the coupling parameter allow synchronization.

\section{Assumptions on the network model}

\label{sec:Dynamical-time-delay-network}

The initial conditions of
\eqref{eq:net_mod_lap} are history functions $\phi_{j}:[-\tau,0]\to\mathbb{R}^{q}$
and we assume that $\phi_{j}$ is continuous for all $j=1,\cdots,n$.
The state space where the solutions $x_j$ and initial conditions $\phi_j$ lives is the
space of continuous functions $C([-\tau,0],\mathbb{R}^q)$ endowed with  
 the maximum norm in $C([-\tau,0],\mathbb{R}^q)$ defined by $\Vert \phi\Vert = \sup_t |\phi(t)|$ where $\vert \cdot \vert$ is some norm in $\mathbb{R}^q$.

The function $f:\mathbb{R}^{q}\rightarrow\mathbb{R}^{q}$ is of class
$C^{r}$, $r\geq2$. Assume that the solutions of the system are uniformly bounded, i.e., the overall system is dissipative. 
So the solution operator is eventually compact and thus we can apply the spectral mapping theorem. 
Thus, we can obtain detailed information by studying the linearized system.

The following assumptions for the vector field $f$  restrict the model \eqref{eq:net_mod_lap} to the synchronization
scenarios that we are interested in.

\begin{assumption} 
\label{assump_f} 
The uncoupled system 
possesses an exponentially locally stable equilibrium $s(t)=s_{0}$.
\end{assumption}

\begin{assumption} 
\label{assump_f-1} 
The uncoupled system 
possesses a locally exponentially  stable periodic solution
$s(t)$. 
\end{assumption}

The coupling function $h:\mathbb{R}^{q}\to\mathbb{R}^{q}$,
is assumed to be bounded and differentiable,
$H=D h(0)$ is the Jacobian matrix at $0$ and $\det(H)\neq0$.

Due to the diffusive nature of the coupling, the coupling term vanishes
identically if the states of all oscillators are identical. This ensures
that the synchronized state $x_{j}(t)=s(t)$ for all $j=1,2,\dots,n$
persists for all coupling strengths $\kappa$.
Let $s(t)$ be the solution specified in Assumptions \ref{assump_f}
or \ref{assump_f-1} and $U\subset\mathbb{R}^{q}$ a tubular neighborhood
which belongs to the basin of attraction of $s(t)$. The set 
\[
\mathcal{S}:=\left\{ x_{j}\in U\subset\mathbb{R}^{q}:\quad j=1,\cdots,n\mbox{ and }x_{1}=\cdots=x_{n}\right\} 
\]
is called the \textit{synchronization manifold}. 
The local stability
of $\mathcal{S}$ determines the stability of the synchronization,
and it depends, in particular, on spectral properties of Laplacian
matrix $L$ defined as $L=D-A$, where $A$ is the adjacency matrix and
$D$ is the diagonal matrix of the network's degrees. 
We assume that the matrix $L$ is diagonalizable.

\section{Main Results}

\label{sec_main_results}

If the uncoupled system has stable equilibrium,
then the coupled system possesses the stable equilibrium
$(s_{0},\dots,s_{0})$ at $\kappa=0$, and the equilibrium solution remains stable for small 
$\kappa$ and small delay by the principle of linearization.

Our result concerns the limit of large delay. In this setting, the principle of linearization does not apply, and the spectral theory of DDE offers the tools to tackle this problem.

Theorem \ref{theo:fp}, gives the interval of the coupling
parameter $\kappa$ where the equilibrium solution of \eqref{eq:net_mod_lap}
is locally exponentially stable for large delays. More specifically, it provides a
critical coupling parameter $\kappa_{c}$ such that the destabilization
occurs for $\kappa>\kappa_{c}$. Moreover,  $\kappa_{c}$
does not depend on the delay.

\begin{theorem}[Persistence of equilibrium]
\label{theo:fp} 
Consider system \eqref{eq:net_mod_lap} and Assumption
\ref{assump_f}. Then there is $\tau_{0}>0$ and
a constant $r_{0}=r_{0}(f,h)>0$ such that for $\kappa\neq0$ and
\begin{equation}
|\kappa|<\kappa_{c}\coloneqq\frac{r_{0}}{\rho_{L}}\label{eq_alpha_condition}
\end{equation}
and $\tau>\tau_{0}$, the equilibrium solution remains locally
exponentially stable. Here, $\rho_{L}$ is the spectral radius of
the Laplacian matrix $L$.

If the condition $\kappa>\kappa_{c}$ is fulfilled, then there exists
such $\tau_{0}>0$ that the equilibrium solution is unstable
for $\tau>\tau_{0}$. 
\end{theorem} 
As follows from Theorem \ref{theo:fp},
the condition $\kappa=\kappa_{c}$ is the strict destabilization value
for $\tau\to\infty$, while the value $\kappa_{c}$ does not depend
on time delay $\tau$. The proof of Theorem \ref{theo:fp} is presented
in Sec.~\ref{sec_spectrum_description}. The interval $(0,r_{0}/\rho_{L})$
of the coupling strength will be called the \emph{synchronization
window}. 

We also show,
in Sec. \ref{sec_transient}, that the characteristic time $t_{\text{tr}}$
in which the trajectories of \eqref{eq:net_mod_lap} approach the
synchronous (equilibrium) solution scales as 
\begin{equation}
t_{\text{tr}}(\kappa) \sim -\tau\ln^{-1}\frac{\kappa}{\kappa_{c}}\label{eq:cor_scaling}
\end{equation}
for $\kappa\nearrow\kappa_{c}$ and large delay. In Sec. \ref{sec_transient}
we derive the estimate (\ref{eq:cor_scaling}) and illustrate it with
an example.

In the case of the synchronous periodic solutions,
the situation is subtle. In particular, one can show that the 
periodic orbits will be 
desynchronized with an increasing delay, however, it is possible to
synchronize them for any finite time-delay. 

The generic condition for synchronization of periodic orbits is a co-dimension one condition. This condition is technical and it will be discussed at length later in Sec. \ref{sec_periodic_orbits}.
To avoid overly technical statements we first give our result for the special case when the isolated dynamics is given by a Stuart-Landau (SL) system. The network model reads
\begin{equation}
\dot{z}_{j}=(\alpha+i\beta)z_{j}-z_{j}|z_{j}|^{2}+\kappa\sum_{l=1}^{n}{A_{jl}}h(z_{l}(t-\tau)-z_{j}(t-\tau)),\label{eq_SL_adj}
\end{equation}
where $z_{j}\in\mathbb{C}$. 

The uncoupled SL system is retrieved if $\kappa=0$. Note that 
it has one equilibrium at the origin, which is asymptotically stable
for $\alpha<0$, and a stable periodic orbit for $\alpha>0$, which
emerges from the Hopf bifurcation at $\alpha=0$.
This model is illustrative since it represents a normal form for a Hopf bifurcation.
\begin{theorem}[Synchronization of periodic SL system]
\label{thm_SL}
Consider the system \eqref{eq_SL_adj} with $\alpha>0$, $\beta>0$ and $h'(0)=1$.
Let the condition
$$\tau \neq \frac{\pi + 2M\pi}{2\beta}, \quad M\in\mathbb{N},$$ be fulfilled. Then,
there is $\tau_0>0$ such that for each $\tau > \tau_0$ there is $\kappa_c = \kappa_c(\tau)$
such that the synchronous
periodic solution is locally exponentially stable
\begin{itemize}
\item[(I)] for $0<\kappa<\kappa_c(\tau)$ if $\cos(\beta\tau)>0$ or
\item[(I)] for $-\kappa_c(\tau)<\kappa<0$ if $\cos(\beta\tau)<0$.
\end{itemize}
Moreover, $\kappa_c(\tau)\to 0$ as $\tau\to\infty$.
\end{theorem}

The Fig. \ref{fig_colormap_maxlamb} will give the numerically computed stability region
from Theorem \ref{thm_SL}.
The full statement will be given Sec. \ref{sec_periodic_orbits}, Theorem \ref{cor_sync_periodc_orbits}.

\begin{remark}
Also, roughly speaking, the critical coupling parameter $\kappa_c(\tau)$ in Theorem \ref{thm_SL}, is written as $\kappa_c(\tau) = r_0(\tau)/\rho_L$.
\end{remark}

The main consequences of the theorems \ref{theo:fp} and \ref{thm_SL} for
different networks are presented as corollaries below. 
From the point of view of the network
structure, the stability of the synchronization manifold is related
to the spectral radius of the Laplacian matrix $\rho_{L}$.
Therefore, fixing $f$, $h$ and $\tau>0$ large enough,
one can study the effect of the changes in the network to the synchronization.

We give the asymptotic behavior of the synchronization window
for two important examples of complex networks, namely, heterogeneous
(e.g. Barab\'asi-Albert (BA) scale-free network),
and homogeneous network (e.g., Erd\"os-R\'eniy (ER) random graph).
More details about BA and ER networks and the proofs of the
corollaries that follow are given in Sec. \ref{sec_sync_loss}.

\begin{corollary}[Synchronization window for large BA networks] 
\label{theorem_BA} 
Consider the model \eqref{eq:net_mod_lap} and its assumptions
\ref{theo:fp} with a BA network with $n$ nodes. Then, for sufficiently
large $n$, the length of the synchronization window scales as $1/\sqrt{n}$.
\end{corollary}

\begin{corollary}[Synchronization window for large ER networks] 
\label{theorem:ER}
Consider the model \eqref{eq:net_mod_lap} and its assumptions  
with a connected ER network with $n$
nodes. Then, for sufficiently large $n$, the length of the synchronization
window scales as $1/\ln n$. Furthermore, the synchronization window
is maximal near the percolation threshold.
\end{corollary}

\section{Preliminaries: variational equation and spectral theory}

\label{sec_spectrum_description}

In this section, we introduce the relevant elements to proof our main results.

\subsection{Dynamics near the synchronization manifold}

In order to study the local stability of the synchronization manifold
we linearize \eqref{eq:net_mod_lap} setting $x_j(t) = s(t) + \eta_j(t)$ with 
$\Vert \eta_j\Vert$ small. We get
\begin{equation}
\label{eq_block_model_aux1}
\dot{\eta}_j(t) = J(t)\eta_j(t) - \kappa \sum_{l=1}^n L_{lj} H \eta_l(t-\tau), \quad j=1,\cdots,n
\end{equation}
where $A_{lj} = g_l \delta_{lj} -  L_{lj}$ with $g_l$ the degree of the node $l$ and 
$\delta_{lj} = 1$
if $l=j$ and $0$ otherwise.
 Here, $J(t)  = D f(s(t))$
is the Jacobian matrix of the vector field $f$ along $s(t)$ and
$H=D h(0)$ is the Jacobian of $h$ at $0$.

Using $\eta = \left(\eta_1^{T},\cdots,\eta_n^{T}\right)$ we rewrite \eqref{eq_block_model_aux1}
as
$$\dot{\eta}(t) = [\mathbb{I}\otimes J(t)]\eta(t) - \kappa (L\otimes H)\eta(t-\tau),$$
where $\otimes$ is the Kronecker product and $\mathbb{I}$ is the identity matrix.
As $L$ is diagonalizable we use $L = R \Lambda R^{-1}$, with 
$\Lambda = \operatorname{diag}(\mu_1,\cdots,\mu_n)$, and the change of coordinates
$\eta= (R^{-1}\otimes \mathbb{I})\xi$ to get the block diagonal equation
\begin{equation}
\label{eq_block_model_aux2}
\dot{\xi}(t) = [\mathbb{I}\otimes J(t)]\xi(t) - \kappa (\Lambda\otimes H)\xi(t-\tau).
\end{equation}
Then, each block of \eqref{eq_block_model_aux2} is given by

\begin{equation}
\dot{\xi}_{j}(t)=J(t)\xi_{j}(t)-\kappa\mu_{j}H\xi_{j}(t-\tau),\quad j=2,\cdots,n.\label{eq_xi}
\end{equation}

Note that $\mu_{1}=0,$ and the variational equation corresponding
to $\mu_{1}$ is 
\[
\dot{\xi}(t)= J(t) \xi(t),
\]
which describes the perturbations within the synchronization manifold.
So, the interest concentrates now on the spectrum of \eqref{eq_xi} for $\mu_2,\cdots,\mu_n$.

It is shown in \cite{Yanchuck2005,Sieber2012,Lichtner2011,YanchukGiacomelli2017}
that the spectrum of \eqref{eq_xi} with large delay consists of two
parts which we introduce here following Ref.~\cite{Sieber2012}.
We consider the following two cases.
\begin{itemize}
\item[--]  Case 1: the synchronized solution $s(t)=s_0$ is an equilibrium
of the uncoupled system 
and then $J(t)=J$ does not depend on $t$. \\
\item[--] Case 2: $s(t)$ is a $T-$periodic solution of the uncoupled system, 
implying that $J(t+T)=J(t)$. 
\end{itemize}

\subsection{Spectral Theory}

Omitting the index $j$
and denoting $\sigma=-\kappa\mu$, Eq.~\eqref{eq_xi} can be rewritten
as 
\begin{equation}
\dot{\xi}(t)=J(t)\xi(t)+\sigma H\xi(t-\tau).\label{eq_periodic_variational_equation_2-1}
\end{equation}

In a general case, the Floquet theory can be used to study the stability
of (\ref{eq_periodic_variational_equation_2-1}) \cite{Hale1993}.
Using the following Floquet-like ansatz 
\begin{equation}
\xi(t)=y(t)e^{\lambda t},\label{eq_Floquet_antsaz-1}
\end{equation}
where $y(t)$ is a non-trivial $T$-periodic function, $\lambda$
and $e^{\lambda T}$ are Floquet exponent and Floquet multipliers
respectively, one can obtain a delay differential equation (DDE) for $y(t)$ 
\begin{equation}
\dot{y}(t)=[J(t)-\lambda\mathbb{I}]y(t)+\sigma e^{-\lambda\tau}Hy(t-\tau).
\label{eq_periodic_equation_tau-1}
\end{equation}

As the function $y(t)$ is $T$ periodic we can rewrite $y(t-\tau)$
in terms of a fraction of the period $T$ so that the large parameter
$\tau$ appears only as a parameter in $e^{-\lambda\tau}$ in Eq.
\eqref{eq_periodic_equation_tau-1}: 
\begin{equation}
\dot{y}(t)=[J(t)-\lambda\mathbb{I}]y(t)+\sigma e^{-\lambda\tau}Hy(t-\eta),\label{eq_periodic_equation_tau_2-1}
\end{equation}
where $\eta=\tau\,\mbox{mod}\,T$. The system \eqref{eq_periodic_equation_tau_2-1}
is said to have a non-trivial periodic solution $y(t)=y(t+T)$ when
it posses a Floquet multiplier equals to $1$.

Let $U$ be the monodromy operator of \eqref{eq_periodic_equation_tau_2-1}. Then,
$U = U(\lambda,\sigma e^{-\lambda\tau})$ and 
$y(t)$ is an eigenvector of $U$ associated with the eigenvalue $1$ (trivial Floquet multiplier),
that is, $[U(\lambda,\sigma e^{-\lambda\tau})-\mathbb{I}]y(t) = 0$. 
This leads to a characteristic equation \cite{Sieber2012}
\begin{equation}
\label{eq_characteristic_H}
\mathcal{H}(\lambda,\sigma e^{-\lambda\tau})=0.
\end{equation}
The solutions $\lambda$ of $\mathcal{H}$ forms the whole spectrum of
\eqref{eq_periodic_variational_equation_2-1}. In the case of $s(t) = s_0$ (equilibrium
solution) then 
$$
\mathcal{H}(\lambda,\sigma e^{-\lambda\tau}) = \det(-\lambda \mathbb{I} +J+\sigma
e^{-\lambda \tau}H)=0.
$$

We will define now certain objects called ``instantaneous spectrum'',
``strongly unstable spectrum'' as well as ``asymptotic continuous
spectrum''. Strictly speaking, these objects do not belong to the
spectrum of the synchronous states, i.e., they are not the Lyapunov
exponents of system (\ref{eq_periodic_variational_equation_2-1}).
However, they play an important role since the spectrum will be well
approximated by the ``strongly unstable spectrum'' and ``asymptotic
continuous spectrum'' as the delay becomes large. Another advantage
of these limiting spectra is that they can be much more easily found
or computed in comparison to the actual spectrum of system (\ref{eq_periodic_variational_equation_2-1}),
see more details in \cite{Yanchuck2005,Yanchuk2009,Yanchuk2010,Lichtner2011,Sieber2012,YanchukGiacomelli2017}.

\begin{definition}[Instantaneous spectrum and strongly unstable spectrum]
\label{def_instantaneous_spectrum} 
The set $\Gamma_{I}$ of all $\lambda\in\mathbb{C}$
for which the linear ODE system 
\[
\dot{y}=[J(t)-\lambda\mathbb{I}]y
\]
has a non-trivial periodic solution $y(t)=y(t+T)$ is called the \emph{instantaneous
spectrum}. The subset $\Gamma_{SU}\subset\Gamma_{I}$ of those $\lambda$
with positive real part is called the \emph{strongly unstable spectrum}.
\end{definition} 
\begin{remark} 
In the case when $J(t)=J$ (does not depend
on time), the instantaneous spectrum $\Gamma_{I}$ consists of the
eigenvalues of $J$. 
\end{remark}

 \begin{definition}[Asymptotic continuous spectrum]
 \label{def_assymptotic_spectrum} 
 For any $\omega\in\mathbb{R},$
the complex number $\lambda=\gamma(\omega)+i\omega$ belongs to the
\emph{asymptotic continuous spectrum} $\Gamma_{A}(\sigma)$ if the
DDE 
\begin{equation}
\dot{y}=[J(t)-i\omega\mathbb{I}]y+\sigma e^{-\gamma-i\phi}Hy(t-\eta)\label{eq:def-pcs}
\end{equation}
has a non-trivial periodic solution $y(t)=y(t+T)$ for some $\phi\in\mathbb{R}$.
Here $\eta=\tau\,\mbox{mod}\,T$. 
\end{definition}

 \begin{remark} If $J(t)$
does not depend on time, the asymptotic continuous spectrum can be
determined from the following equation 
\begin{equation}
\det\left[J-i\omega\mathbb{I}+\sigma e^{-\gamma-i\phi}H\right]=0.\label{eq:cheqeq}
\end{equation}
\end{remark} As follows from \cite{Yanchuck2005,Yanchuk2009,Yanchuk2010,Heiligenthal2011,Lichtner2011,Sieber2012,YanchukGiacomelli2017},
the spectrum of generic linear delay system (\ref{eq_periodic_variational_equation_2-1})
converges to either the instantaneous spectrum or to the curves in
the complex plane $\gamma(\omega)/\tau+i\omega,$ $\omega\in\mathbb{R}$,
where $\gamma(\omega)$ is defined from the asymptotic continuous
spectrum (note the division by $\tau$). The second part of the spectra
-- consisting of eigenvalues with asymptotically vanishing real parts
-- approaches a continuous curves asymptotically, while being still
discrete for any finite $\tau$; for this reason, it is called \emph{pseudo-continuous
spectrum}.

An example of a typical spectrum of the system with long delay is
shown in Fig.~\ref{fig:Numerically-computed-spectrum}. In particular,
the distances between neighboring eigenvalues within one curve of
the pseudo-continuous spectrum scale as $2\pi/\tau$ for large delay,
and in the limit of $\tau\to\infty$ they vanish and the eigenvalues
fill the curve \cite{Sieber2012}.

The union of the strongly unstable spectrum, $\Gamma_{SU}$, with
the asymptotic continuous spectrum $\Gamma_{A}$ forms the approximation
of the whole spectrum of Eq.~\eqref{eq_xi}, given that the instantaneous
spectrum does not contain eigenvalues with zero real parts and some
non-degeneracy conditions are fulfilled \cite{Lichtner2011,Sieber2012}.
Moreover, if the set of the strongly unstable spectrum is empty ($\Gamma_{SU}=\emptyset$),
and the asymptotic continuous spectrum is entirely contained on the
left side of the complex plane ($\Gamma_{A}\subset\mathbb{C}_{-}$)
then the trivial solution of Eq.~\eqref{eq_xi} is exponentially
asymptotically stable.

\section{Proof of the main theorems}

In this section we use the given description of the spectrum
of a general linear DDE to prove Theorems \ref{theo:fp} (persistence of equilibrium solution), 
and \ref{cor_sync_periodc_orbits}  (synchronization of periodic orbit).

The result obtained in Theorem \ref{theo:fp} could also be proved by using
 Lyapunov-Kravoskii functionals \cite{FRIDMAN2014271}. However, we use
the approach of studying the spectrum for large delay  
  not only because it is powerful, but it also serves
as a preparation to prove the result of synchronization of periodic solutions.

\subsection{Persistence of equilibrium solution: Theorem \ref{theo:fp}}

\label{sec_equilibrum_case}

In this section we specify the stability condition for the
case when system \eqref{eq:net_mod_lap} possesses a synchronous solution
$s_0,$ i.e. $f(s_0)=0$, and the Jacobian $J=D f(s_0)$
is a constant $q\times q$ matrix. The corresponding characteristic
equation is 
\begin{equation}
\det\left[-\lambda\mathbb{I}+J+\sigma He^{-\lambda\tau}\right]=0.\label{eq:ch_eq_eq}
\end{equation}
The real parts of the eigenvalues $\lambda$ determine the stability,
and, as it was discussed above, the whole spectrum converges to either
the strongly unstable spectrum or pseudo-continuous spectrum. Lemma
\ref{lemma_pseudo_spectrum} gives explicit dependence of the asymptotic
continuous spectrum on the coupling strength $\sigma$.

\begin{lemma}[Ref.\cite{Yanchuck2005}]
\label{lemma_pseudo_spectrum}
Consider the linear delay differential equation 
\begin{equation}
\dot{\xi}(t)=J\xi(t)+\sigma H\xi(t-\tau)\label{eq_xi_general}
\end{equation}
with $J,H\in\mathbb{R}^{q\times q}$ and $\det(H)\neq0$. Then, the
asymptotic continuous spectrum for (\ref{eq_xi_general}) is given
by the branches 
\begin{equation}
\{\gamma_{l}(\omega,\sigma)+i\omega\in\mathbb{C}:\,\gamma_{l}(\omega,\sigma)=-\ln\vert g_{l}(\omega)\vert+\ln|\sigma|\},\label{eq_branches_fix_point}
\end{equation}
where $l=1,\cdots,q$ and $g_{l}$ are complex roots of the polynomial
\begin{equation}
\det\left[-i\omega\mathbb{I}+J+gH\right]=0.\label{eq:ch-eq}
\end{equation}
\end{lemma} 

The application of Lemma \ref{lemma_pseudo_spectrum} to
system \eqref{eq_xi} leads to the proof of Theorem \ref{theo:fp}.

\begin{proof}[Proof of Theorem \ref{theo:fp} (Persistence of equilibrium solution)]
The basic idea of the proof is to show that the strongly unstable
spectrum is empty and the asymptotic continuous spectrum is stable
if and only if the condition (\ref{eq_alpha_condition}) is satisfied.
The rest of the proof
elaborates this idea more carefully in details.

The instantaneous spectrum $\Gamma_{I}$ coincides with the spectrum
of $J$ and, by the Assumption \ref{assump_f}, the equilibrium solution
of the isolated system is asymptotically stable, thus there are no
eigenvalues of $J$ with positive real parts, hence $\Gamma_{SU}=\emptyset$.

Using Lemma \ref{lemma_pseudo_spectrum}
for the variational Eq. \eqref{eq_xi}. Then, the real part of the
asymptotic continuous spectrum is 
\begin{equation}
\gamma_{l,j}(\omega,\sigma)=-\ln\vert g_{l}(\omega)\vert+\ln\vert \kappa \mu_{j}\vert.\label{eq_gamma_assimpt_spectrum}
\end{equation}
The condition $\gamma_{l,j}(\omega,\sigma)<0$ is fulfilled if and
only if 
\begin{equation}
\vert \kappa \mu_{j}\vert<\vert g_{l}(\omega)\vert.\label{ineq_alphamuk}
\end{equation}
The asymptotic continuous spectrum lies strictly on the left side
of the complex plane if the inequality \eqref{ineq_alphamuk} holds
for all $j=2,\cdots,n$, $l=1,\cdots,q$, and $\omega\in\mathbb{R}$.
This leads to the condition 
\begin{equation}
\vert \kappa\vert <\dfrac{\min_{l}\inf_{\omega}\vert g_{l}(\omega)\vert}{\max_{j}\vert\mu_{j}\vert}=\frac{r_{0}}{\rho_{L}},\label{eq:111}
\end{equation}
where $r_{0}=\min_{l=1,\cdots,q}\inf_{\omega\in\mathbb{R}}\vert g_{l}(\omega)\vert$
exists and is always bounded from zero. Indeed, if we assume the opposite,
i.e., $r_{0}=0$, then it means that there exist such $l_{0}$ and
$\omega_{0}$ that $g_{l_{0}}(\omega_{0})=0$, and hence, $\det(-i\omega_{0}+J)=0$
implying that $i\omega_{0}$ belongs to the spectrum of $J$. In this
way we arrive at the contradiction to the assumption that the spectrum
of $J$ has strictly negative real parts. Moreover, $r_{0}=r_{0}(J,H)$
since the functions $g_{l}(\omega)$ are the roots of the characteristic
equation \eqref{eq:ch-eq}. The number $\rho_{L}=\max_{j=2,\cdots,n}\vert\mu_{j}\vert$
stands for the spectral radius of the Laplacian matrix $L$. Therefore,
under the condition \eqref{eq_alpha_condition}, the asymptotic continuous
spectrum $\Gamma_{A}\subset\mathbb{C}_{-}$ and therefore the zero
solution of \eqref{eq_xi} is stable.

Having both $\mbox{Re}\,\Gamma_{A}<0$ and $\Gamma_{SU}=\emptyset$,
Theorem \ref{thm:Jan} implies that there exists such $\tau_{0}>0$
that for all $\tau>\tau_{0}$ the equilibrium $s_0$ is locally
exponentially stable under the condition (\ref{eq:111}).

The statement about the instability of the equilibrium follows from
the fact that for $\vert \kappa\vert >\kappa_{c}$ there exist such $l\in\{1,\dots,q\}$,
$j\in\{2,\dots,n\}$ and $\omega\in\mathbb{R}$ that $\vert \kappa \mu_{j}\vert>\vert g_{l}(\omega)\vert$,
and, hence the real part of the asymptotic continuous spectrum is
positive $\gamma_{l,j}(\omega,\sigma)>0.$ Therefore, for sufficiently
large time-delays and for $\vert \kappa\vert >\kappa_{c}$, there will be eigenvalues
from the pseudo-continuous spectrum with positive real parts. 
\end{proof}

With the result of persistence of equilibrium solution proven, we move to the
case of periodic synchronous solution.

\subsection{Synchronization of periodic orbits}

\label{sec_periodic_orbits}

We have previously stated our result on the synchronization of periodic orbits for the SL oscillators (Theorem \ref{thm_SL}),
so, we could avoid the discussion about the characteristic equation $\mathcal{H}$. 
Now we present a generalization of the announced result  and we prove it in this section.

\begin{theorem}[Synchronization of periodic orbits for any finite delay] \label{cor_sync_periodc_orbits}
Consider
system \eqref{eq:net_mod_lap} and Assumption
\ref{assump_f-1}. Let an additional non-degeneracy assumption $\partial_{2}\mathcal{H}(0,0)\ne0$
be fulfilled, where $\mathcal{H}$ is defined by \eqref{eq_characteristic_H}.
Then, given $\tau$ fixed large enough, there exists $\kappa_{c}(\tau)>0$ such that one of the two
following statements hold: 
\begin{itemize}
\item[(I)] The synchronous periodic solution is locally exponentially 
stable for $0<\kappa<\kappa_{c}$ and unstable for $-\kappa_{c}<\kappa<0$. 
\item[(II)] The synchronous periodic solution is locally exponentially 
stable for $-\kappa_{c}<\kappa<0$ and unstable for $0<\kappa<\kappa_{c}$. 
\end{itemize}
Moreover, $\kappa_{c}\to0$ as $\tau\to\infty$. That is, for any $\kappa>0$, there exists
$\tau_0>0$ such that for all $\tau>\tau_0$ the synchronous periodic solution is unstable.
\end{theorem}

\begin{remark}
The co-dimension 1 condition $\partial_{2}\mathcal{H}(0,0)=0$ may
lead to isolated points, where the synchronization of the synchronous
periodic orbit cannot be achieved (see example in Sec.~\ref{subsec:Synchronous-Periodic-Orbit}).
\end{remark}

The proof of Theorem \ref{cor_sync_periodc_orbits} rely on the  
 following lemma, adapted from \cite{Sieber2012}
for the case of the stability of a periodic
synchronized solution $s(t)$ with respect to desynchronized perturbations
(transverse to the synchronization manifold).

\begin{lemma} 
\label{thm:Jan}
The synchronous periodic orbit $s(t)$ of system (\ref{eq:net_mod_lap})
with period $T$ is exponentially  stable with respect to
perturbations transverse to the synchronization manifold for all sufficiently
large $\tau$ if all of the following conditions hold: 
\begin{itemize}
\item[S-I] (No strong instability) All elements of the instantaneous spectrum
$\Gamma_{I}$ have negative real parts (this implies in particular
that the strongly unstable spectrum $\Gamma_{SU}$ is empty), 
\item[S-II] (Simple trivial multiplier) The trivial Floquet exponent $\lambda=0$
is simple. 
\item[S-III] (Weak stability) the asymptotic continuous spectrum $\Gamma_{A}(\sigma)$
is contained in $\left\{ \lambda\in\mathbb{C}:\,\Re(\lambda)<0\right\} $
for all $\sigma\in\left\{ -\kappa\mu_{2},\dots-\kappa\mu_{n}\right\} $
with $\mu_{j}$, $j=2,\dots,n$ being all nontrivial eigenvalues of
the Laplacian matrix $L$. 
\end{itemize}
And, it is exponentially unstable with respect
to perturbations transverse to the synchronization manifold for sufficiently
large $\tau$ if one of the following conditions hold:
\begin{itemize}
\item[U-I] (Strong instability) the strongly unstable spectrum $\Gamma_{SU}$
is non-empty, or 
\item[U-II] (Weak instability) a non-empty subset of the asymptotic continuous
spectrum $\Gamma_{A}(-\kappa\mu_{j})$ has positive real part for
nontrivial eigenvalue $\mu_{j}$ of the Laplacian matrix $L$. 
\end{itemize}
\end{lemma} 

\begin{proof} 
This theorem follows almost directly
from Theorem 6 of \cite{Sieber2012}, and we highlight the main features for our setup.

First of all, the existence of the periodic solution $s(t)$ does
not depend on time delay $\tau$. As a result, $\tau$ can be considered
as a continuous parameter here, instead of the parameter $N$ which
is an integer part of $\tau/T$ used in \cite{Sieber2012}. Hence,
the asymptotic statements with respect to $N$ are equivalently formulated
for $\tau$ in the present theorem.

Secondly, the stability of the periodic solution is determined by
the variational equation (\ref{eq_xi}) that splits into the part
along the synchronization manifold with $\mu_{1}=0$ and the remaining
part for the transverse perturbations with nonzero $\mu_{j}$. Hence,
the transverse stability is determined by the variational equation
(\ref{eq_periodic_equation_tau-1}) with $\sigma\in\left\{ -\kappa\mu_{2},\dots-\kappa\mu_{n}\right\} $.
The conditions for the stability S-I -- S-III guarantee that the spectrum
of the corresponding variational equations is stable for large enough
$\tau$. \end{proof}

The linear stability of a synchronous
periodic orbit is governed by the linearized equation \eqref{eq_periodic_variational_equation_2-1}
with periodic $J(t)$. In Ref. \cite{Sieber2012} the authors proved
results about the stability of this equation for large $\tau$. In
particular, the existence of a holomorphic function $\mathcal{H}:\,\Omega_{1}\times\Omega_{2}\subseteq\mathbb{C}\times\mathbb{C}\to\mathbb{C}$
is shown, such that $\lambda$ is a Floquet exponent of the linear
DDE \eqref{eq_periodic_variational_equation_2-1} if and only if 
$
\mathcal{H}\left(\lambda,\sigma e^{-\lambda\tau}\right)=0
$ (see Eq. \eqref{eq_characteristic_H}).

The function $\mathcal{H}$ in this case is analogous to the characteristic
equation (\ref{eq:ch_eq_eq}) for the case of equilibrium. The main
difference of the periodic case is that the function $\mathcal{H}$
is not given explicitly. The asymptotic continuous spectrum is determined
from the equation 
\begin{equation}
\mathcal{H}\left(i\omega,\sigma e^{-\gamma}e^{-i\phi}\right)=0,\label{eq:psdf-sigma}
\end{equation}
(compare (\ref{eq:psdf-sigma}) and (\ref{eq:cheqeq})). Here we should
emphasize that, in contrast to Ref.~\cite{Sieber2012}, we explicitly
write the parameter $\sigma$ in the argument of the function $\mathcal{H}$
since the dependence on $\sigma$ is of interest for this study. Moreover,
this parameter can be eliminated from Eq.~\eqref{eq:psdf-sigma}
by the transformation 
\[
\gamma_{(1)}+i\phi_{(1)}:=\gamma+i\phi-\ln|\sigma|-i\phi_{\sigma}
\]
leading to the equation 
\[
\mathcal{H}\left(i\omega,e^{-\gamma_{(1)}}e^{-i\phi_{(1)}}\right)=0
\]
which is the same as Eq.~\eqref{eq:psdf-sigma} but with $\sigma=1$.
Here we denoted $\sigma=|\sigma|e^{i\phi_{\sigma}}$.
As a result, the following Lemma holds: 
\begin{lemma}
 The point $\gamma+\ln|\sigma|+i\omega\in\mathbb{C}$
belongs to the asymptotic continuous spectrum of system (\ref{eq_periodic_variational_equation_2-1})
if and only if the point $\gamma+i\omega$ belongs to the asymptotic
continuous spectrum of system (\ref{eq_periodic_variational_equation_2-1})
with $\sigma=1$. 
\end{lemma} 

\begin{remark} 
The main effect of the feedback
strength $\sigma$ for strong delays is the shift of the asymptotic
continuous spectrum by the value $\ln|\sigma|$. 
\end{remark}

\begin{proof}[Proof of Theorem \ref{cor_sync_periodc_orbits}]
The proof consists on showing that the conditions S-I -- S-III of Lemma \ref{thm:Jan}
are satisfied.

We remark that 
\[
\mathcal{H}(0,0)=0
\]
since $\mathcal{H}(\lambda,0)$ is the characteristic equation of
the uncoupled system, which possesses one simple trivial multiplier
by Assumption \ref{assump_f}. Consider now the implicit function
problem 
\begin{equation}
\mathcal{H}(i\omega,Y)=0.\label{eq_h_periodic_characteristic}
\end{equation}
It has a unique smooth solution $Y=Y(\omega)$ with $Y(0)=0$ provided
$\partial_{2}\mathcal{H}(0,0)\ne0$. As a result, the asymptotic continuous
spectrum 
\begin{equation}
\gamma_{(1)}(\omega)=-\ln\left|Y(\omega)\right|\label{eq_gamma_1}
\end{equation}
has a singularity at $\omega=0$ and $\lim_{\omega\to0}\gamma_{(1)}(\omega)=\infty$.
Hence, the asymptotic continuous spectrum of (\ref{eq_periodic_variational_equation_2-1})
is singular as well 
\[
\lim_{\omega\to0}\gamma(\omega)=\infty
\]
independently of the value of $\sigma$. Hence, the asymptotic continuous
spectrum possesses always an unstable part, which implies instability
for large enough values of time-delays. 

We know that
for $\sigma=0$ the characteristic equation $\mathcal{H}(\lambda,0)=0$
possesses one simple zero root $\lambda=0$, hence $\partial_{1}\mathcal{H}(0,0)\ne0$.
We find how the real parts of these eigenvalues change if $\sigma$
deviates from zero using the implicit function theorem for the equation
$\mathcal{H}(\lambda(\sigma),\sigma e^{-\lambda(\sigma)\tau})=0$.
Since 
\[
\mathcal{\partial_{\lambda}H}(\lambda,\sigma e^{-\lambda\tau})|_{\sigma=0,\lambda=0}=\partial_{1}\mathcal{H}(0,0)\ne0,
\]
we have the unique solution $\lambda(\sigma)$ for small $\sigma$
and $\partial_{\sigma}\lambda|_{\sigma=0}=-\partial_{2}\mathcal{H}(0,0)/\partial_{1}\mathcal{H}(0,0)$.
Here $\partial_{2}\mathcal{H}(0,0)\ne0$ by assumption. Hence 
\begin{equation}
\partial_{\sigma}\Re\left(\lambda\right)|_{\sigma=0}=-\Re\left[\frac{\partial_{2}\mathcal{H}(0,0)}{\partial_{1}\mathcal{H}(0,0)}\right]=:\alpha_{P}.\label{eq_alpha_P}
\end{equation}
In particular, this expression shows that for $\alpha_{P}>0$, the
periodic solution will be destabilized for positive $\sigma$, and
stabilized for negative $\sigma$. For $\alpha_{P}<0$, the stabilization
occurs for positive $\sigma$ and destabilization for negative. It
is important to notice, that the stabilization (or destabilization)
occurs for all transverse modes simultaneously, since the eigenvalues
$\mu_{2},\dots,\mu_{n}$ of the Laplacian matrix are positive and
$\sigma$ can admit values $-\kappa\mu_{j}$, which have the same
sign for all $\mu_{j}$, $j=2,\dots,n$.

\end{proof}

\section{Example: Coupled Stuart-Landau systems\label{subsec:Example:-The-Stuart-Landau}}

\label{sec_example}

Let us consider an example of the ring of coupled Stuart-Landau (SL)
oscillators shown in Fig.~\ref{fig:grafofc}.

\begin{figure}[!ht]
\centering \begin{tikzpicture}

\tikzset{vertex/.style = {shape=circle,draw,minimum size=1.5em}}
\tikzset{edge/.style = {->,> = latex'}}
\node[vertex] (1) at  (0,1) {$1$};
\node[vertex] (2) at  (-2,1) {$2$};
\node[vertex] (3) at  (-2,-1) {$3$};
\node[vertex] (4) at  (0,-1) {$4$};

\draw[edge][bend right] (1) to (2);
\draw[edge][bend right] (2) to (3);
\draw[edge][bend right] (3) to (4);
\draw[edge][bend right] (4) to (1);

\end{tikzpicture} \caption{\label{fig:grafofc} A directed ring network with 4 nodes.}
\end{figure}
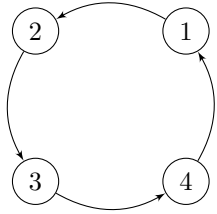

The dynamics of a node $j$ 
is given by \eqref{eq_SL_adj}. Throughout this section, we use the identity as the coupling function. We rewrite Eq. \eqref{eq_SL_adj},
in terms of the Laplacian matrix as 
\begin{align}
\dot{z}_{j}=(\alpha+i\beta)z_{j}-z_{j}|z_{j}|^{2}-\kappa\sum_{l=1}^{n}{L_{jl}}z_{l}(t-\tau)\label{eq_SL}
\end{align}

Figure \ref{fig_pcs_and_sus} shows the spectrum of
two coupled Stuart-Landau oscillators linearized at the origin and
coupling function being the identity.

\begin{figure}[!ht]
\centering \includegraphics[scale=0.4]{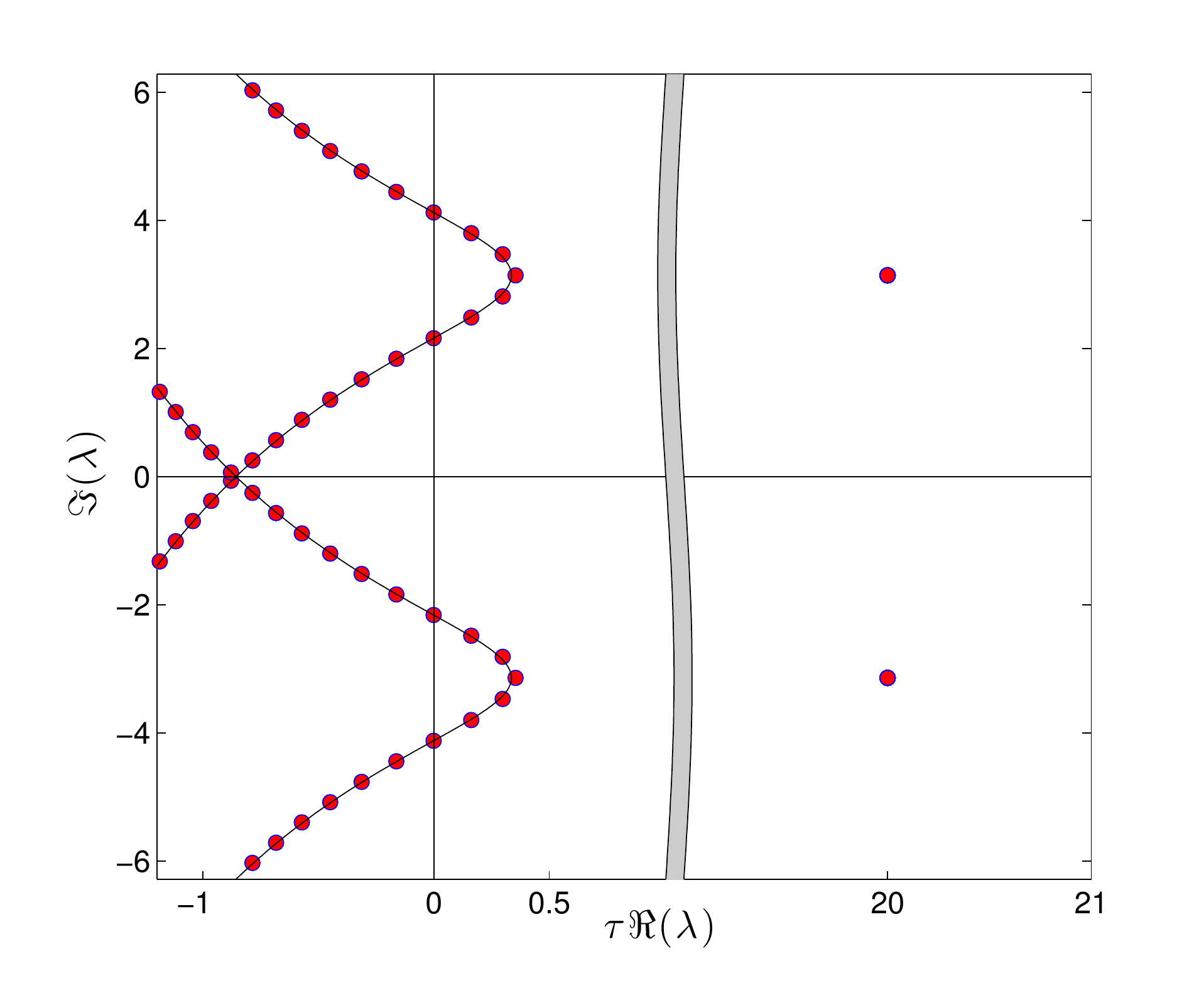}
\caption{\label{fig:Numerically-computed-spectrum}Numerically computed spectrum
for the equilibrium of two Stuart-Landau oscillators (see Sec.~\ref{subsec:Example:-The-Stuart-Landau})
coupled as in \eqref{eq:net_mod_lap} with identity coupling function
and parameters $\alpha=1$, $\beta=\pi$, $\tau=20$, and $\kappa=0.7$.
The points approaching the curve on the left side of the figure belongs
to the pseudo-continuous spectrum and the isolated points on the right
belongs to the strongly unstable spectrum. Solid lines show the re-scaled
asymptotic continuous spectrum $\Gamma_{A}$. The gray strip represents
a break on the figure, which is necessary due to the different scales
of the two parts of the spectrum.}
\label{fig_pcs_and_sus} 
\end{figure}

\subsection{Persistence of equilibrium}

Firstly, we consider the case $\alpha<0$ (equilibrium solution) so
that the eigenvalues of the uncoupled system have negative real parts.
The Jacobian matrix at zero $z=0$ has the eigenvalues $\alpha\pm i\beta$.
Therefore, the strongly unstable spectrum is empty, and the whole
spectrum of the zero equilibrium for the network system \eqref{eq_xi}
consists only of the asymptotic-continuous spectrum for long delays.

The laplacian matrix of the network in Fig. \ref{fig:grafofc} has
spectral radius $\rho_{L}=2$. In order to compute the value $r_{0}$
from the condition (\ref{eq_alpha_condition}) of Theorem \ref{theo:fp}
we compute the functions $g_{l}(\omega)$, $l=1,2$ using Eq.~(\ref{eq:ch-eq}):
\[
g_{1,2}(\omega)=i(\omega\pm\beta)-\alpha.
\]
Hence, we find $r_{0}=\min_{l}\inf_{\omega}\vert g_{l}(\omega)\vert=|\alpha|$.
Therefore, the synchronization manifold for the considered network
is locally exponentially stable if and only if 
\begin{equation}
0<\kappa<\kappa_{c}=\frac{|\alpha|}{2}\label{eq:sl-syn}
\end{equation}
for sufficiently long time delay $\tau$ (and unstable if $\kappa>\kappa_{c}$).

Fig. \ref{fig:Time-series-of} illustrates the convergence of the
trajectories to the equilibrium (left panel) for the case
when the condition (\ref{eq:sl-syn}) is fulfilled and the absence
of convergence in the opposite case. The detailed parameter values
are given in the figure's caption. Moreover, in order to compute the
synchronization error we compare each solution to a fixed one taking
the norm of maximal difference, that is, we measure the synchronization
error by $\max_{j}\Vert z_{1}(t)-z_{j}(t)\Vert$.

\begin{figure}[ht]
\centering 
\noindent\makebox[1\textwidth]{
\includegraphics[scale=0.37]{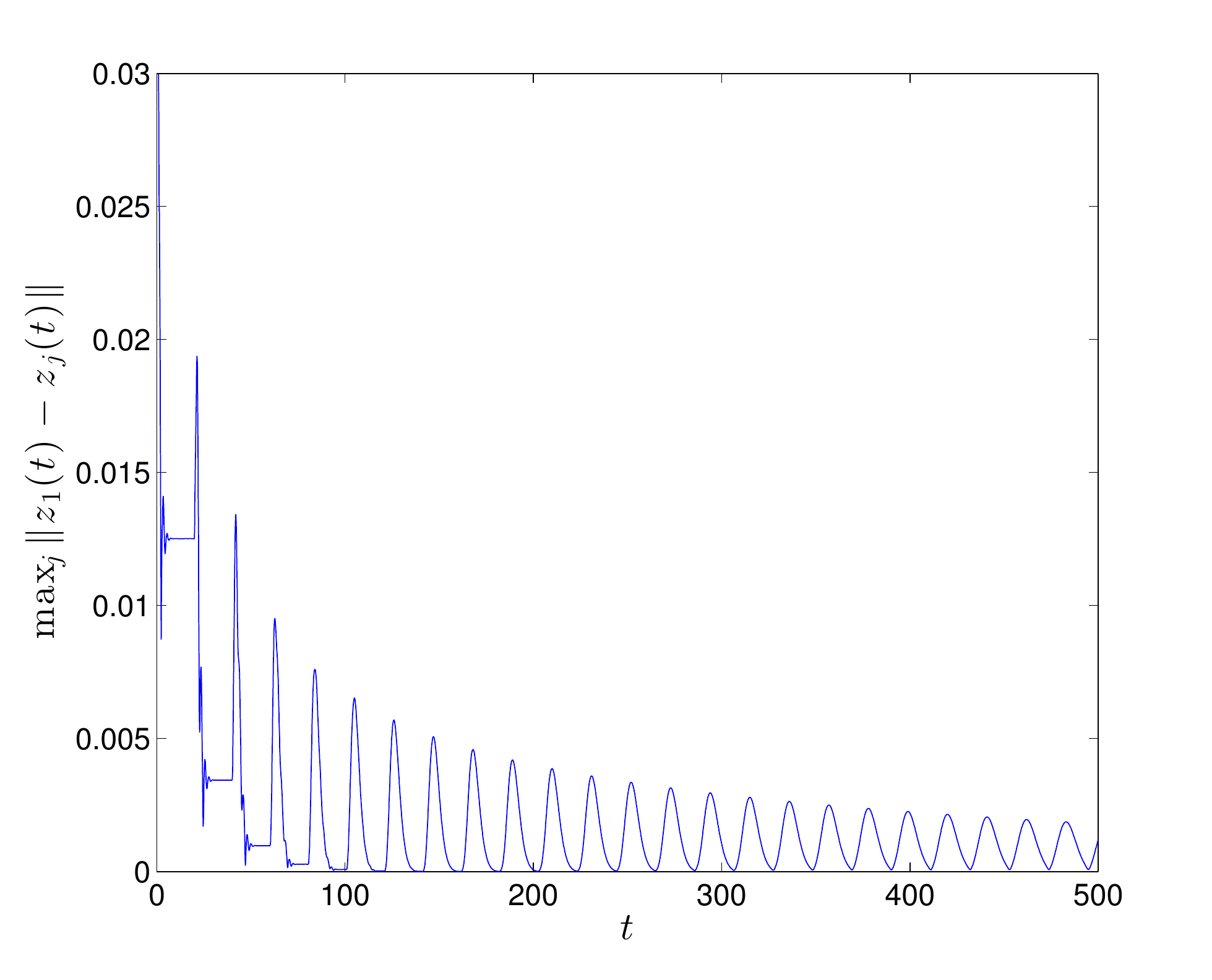} \includegraphics[scale=0.37]{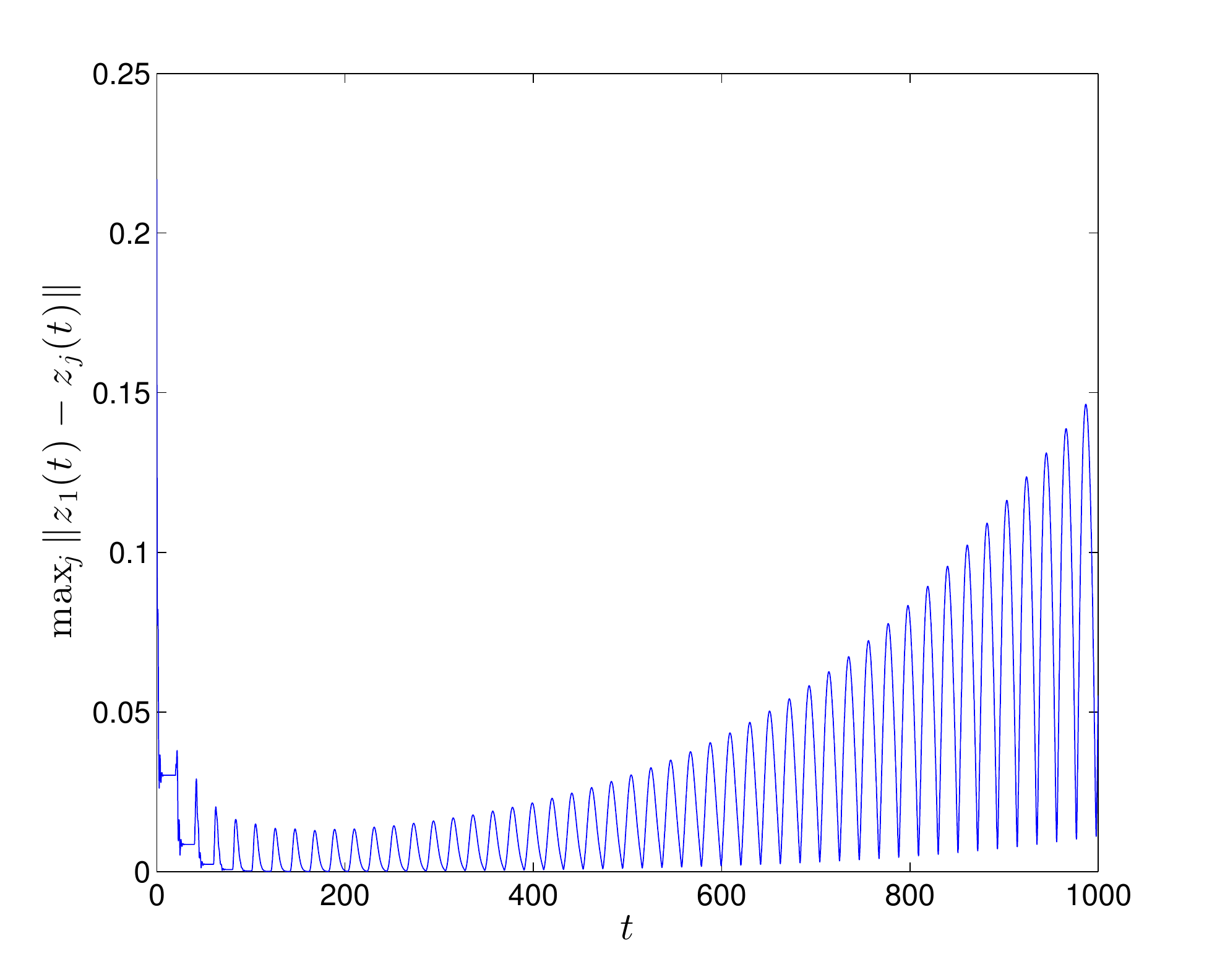} 
}

\noindent \caption{\label{fig:Time-series-of} Time series of the synchronization error
for Eq. \eqref{eq_SL_adj} and network given in Fig. \ref{fig:grafofc}
where $\kappa=0.49$ for the left figure and $\kappa=0.51$ for the
right. Other parameters are $\alpha=-1$, $\beta=\pi$, $\tau=100$.
The history functions were taken as constant and non-zero. }
\end{figure}

\subsection{Periodic orbit synchronization \label{subsec:Synchronous-Periodic-Orbit}}

Let us consider the case $\alpha>0$ when the synchronous periodic
solution $\sqrt{\alpha}e^{i\beta t}$ exists. In this section we prove Theorem \ref{thm_SL}.

Consider the transformation $z(t)=r(t)e^{i\beta t}$, then Eq. \eqref{eq_SL}
reads, in terms of $r(t)$, 
\begin{align}
\dot{r}_{j}=\left(\alpha-|r_{j}|^{2}\right)r_{j}-\kappa e^{-i\beta\tau}\sum_{l=1}^{n}L_{jl}r_{j}(t-\tau).\label{eq_y_periodic}
\end{align}
Note that the solution $z(t)=\sqrt{\alpha}e^{i\beta t}$ is transformed
into a family of equilibria. The variational equation of \eqref{eq_y_periodic}
(considering real and imaginary parts) around the equilibrium solution
$r=\sqrt{\alpha}$, equivalent to \eqref{eq_xi}, is 
\begin{align}
\dot{\xi}_{j}(t)=\begin{pmatrix}-2\alpha & 0\\
0 & 0
\end{pmatrix}\xi_{j}(t)-\kappa\mu_{j}\begin{pmatrix}\cos(\beta\tau) & -\sin(\beta\tau)\\
\sin(\beta\tau) & \cos(\beta\tau)
\end{pmatrix}\xi_{j}(t-\tau).\label{eq_xi_SL_polar}
\end{align}
where $\mu_{j}\geq0$ are the eigenvalues of $L$. For short, we write
Eq. \eqref{eq_xi_SL_polar} as $\dot{\xi}_{j}(t)=J_{0}\xi_{j}(t)-\kappa\mu_{j}\mathcal{T}\xi_{j}(t-\tau)$.

The instantaneous spectrum can be computed using Definition \ref{def_instantaneous_spectrum}.
It consists of the eigenvalues of the non-delayed part of \eqref{eq_xi_SL_polar}
and reads $\Gamma_{I}=\left\{ -2\alpha,0\right\} $. The eigenvalue
$0$ in the instantaneous spectrum is associated with the trivial
Floquet multiplier producing a singularity in the asymptotic spectrum.
Then, as predicted by Theorem \ref{cor_sync_periodc_orbits},
the synchronization cannot be achieved for an arbitrary large delay.
However,  Theorem \ref{cor_sync_periodc_orbits} guarantee the existence of a
 synchronization interval $(0,\kappa_{c})$ (or $(\kappa_c,0)$ with $\kappa_c<0$) with the decreasing
length with increasing $\tau$.

The asymptotic continuous spectrum can be computed as in the case
of equilibrium solution. Using Eq. \eqref{eq:cheqeq}, one obtains
$\gamma_{j,l}(\omega)=-\ln|g_{l}(\omega)|+\ln|\kappa\mu_{j}|$ where
$g_{l}(\omega)$ are the solutions of 
\begin{equation}
\det\left[-i\omega\mathbb{I}_{2}+J_{0}+g\mathcal{T}\right]=0.\label{eq_charc_g_periodic}
\end{equation}
Note that Eqs. \eqref{eq_h_periodic_characteristic} and \eqref{eq_charc_g_periodic}
are the same. Hence, we write $$\mathcal{H}(i\omega,g)=\det\left[-i\omega\mathbb{I}_{2}+J_{0}+g\mathcal{T}\right]$$
or, more specifically, 
\[
\mathcal{H}(i\omega,g)=g^{2}-2\cos(\beta\tau)(\alpha+i\omega)g+2\alpha\omega i-\omega^{2}.
\]
The non-degeneracy condition $\partial_{2}\mathcal{H}(0,0)\neq0$
(assumption in Theorem \ref{cor_sync_periodc_orbits})
is satisfied since $\partial_{2}\mathcal{H}(0,0)=-2\alpha\cos(\beta\tau)\neq0$
provided $\tau\neq(\pi+2M\pi)/(2\beta)$, $M\in\mathbb{N}$.  

The solution of \eqref{eq_charc_g_periodic} reads 
\begin{align}
g_{\pm}(\omega)=(\alpha+i\omega)\cos(\beta\tau)\pm\left[\alpha^{2}\cos^{2}(\beta\tau)+(w^{2}-2\alpha\omega i)\sin^{2}(\beta\tau)\right]^{1/2}.\label{eq_assimp_periodic}
\end{align}
From \eqref{eq_assimp_periodic}, we get $g_{-}(0)=0$ and $g_{+}(0)=2\alpha\cos(\beta\tau)\neq0$
provided $\tau\neq(\pi+2M\pi)/(2\beta)$, $M\in\mathbb{N}$. Therefore,
the singularity occurs uniquely at the function $g_{-}(\omega)$.

In order to obtain $\alpha_{P}$, as introduced in
Eq. \eqref{eq_alpha_P}, we compute $\partial_{1}\mathcal{H}(0,0)=2\alpha$.
Hence, $\alpha_{P}=\cos(\beta\tau)$. As $\sigma=-\kappa\mu_{j}$,
this implies that synchronization occurs for negative $\sigma$ when
$\cos(\beta\tau)>0$ and for positive $\sigma$ otherwise (positive
$\sigma$ is obtained if $\kappa<0$). 

We can also compute the spectrum of \eqref{eq_xi_SL_polar} by using
Eq. \eqref{eq:ch_eq_eq}. So, we get the transcendental equation 
\begin{align}
\lambda^{2}-2\left(\sigma\cos(\beta\tau)e^{-\lambda\tau}-\alpha\right)\lambda-\sigma e^{-\lambda\tau}\left(2\cos(\beta\tau)-\sigma e^{-\lambda\tau}\right)=0\label{eq_transcedental_SL_periodic}
\end{align}
with $\sigma=-\kappa\mu_{j}$. Eq. \eqref{eq_transcedental_SL_periodic}
can be numerically solved.

We compare the solutions of \eqref{eq_transcedental_SL_periodic}
to the analytical approximation. The asymptotic spectrum (given in
terms of Eq. \eqref{eq_assimp_periodic}) and the spectrum points
(solutions of Eq. \eqref{eq_transcedental_SL_periodic}) are given
in Fig. \ref{fig_pseudo_and_assympt_spectrum_SL_periodic_tau20} with
parameters given in the figure's caption. 

\begin{figure}[!ht]
\centering \includegraphics[scale=0.4]{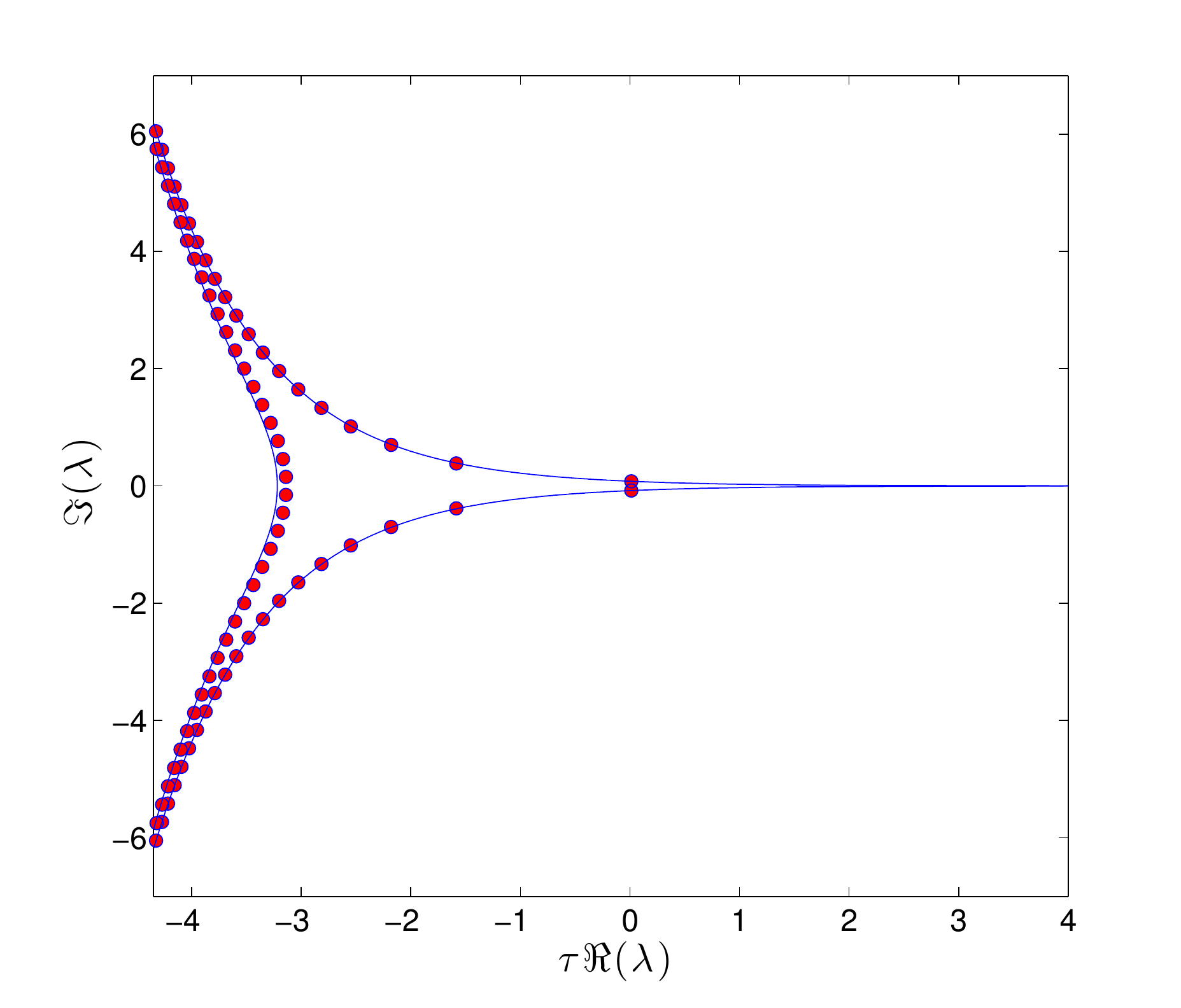}
\caption{{The asymptotic continuous spectrum (blue lines) and the pseudo-continuous
spectrum (red dots) for the periodic solution of Stuart-Landau system
given by Eqs. \eqref{eq_assimp_periodic} and \eqref{eq_transcedental_SL_periodic}
respectively. The parameters are $\sigma=-0.08$ (with $\mu=\rho_{L}=2$,
the spectral radius of the Laplacian matrix of the network in Fig.
\ref{fig:grafofc}, and $\kappa=0.04$), $\alpha=1$, $\beta=\pi$
and $\tau=20$.}}
\label{fig_pseudo_and_assympt_spectrum_SL_periodic_tau20} 
\end{figure}

From the data of Fig. \ref{fig_pseudo_and_assympt_spectrum_SL_periodic_tau20}
we notice that for $\alpha=1$, $\beta=\pi$ and $\tau =20$, stale synchronization is attained
for $0<\kappa\lesssim 0.04$.

Now, in order to have a complete view of the stability domain, we produce
a synchronization map in the parameter space $\sigma\times \tau$.
The Fig. \ref{fig_colormap_maxlamb} shows such a color
map presenting the values of $\max_m\Re(\lambda_m)<0$, where $\lambda_m$, $m\in\mathbb{Z}$, is a
solution of Eq. \ref{eq_transcedental_SL_periodic}, in the $\sigma\times \tau$
parameter space. Note that $\sigma=-\kappa\rho_{L}$ and, as predicted
by Theorem \ref{cor_sync_periodc_orbits}, the stable synchronization
either occurs for $0<\kappa<\kappa_{c}$ or for $-\kappa_{c}<\kappa_{c}<0$
with $\kappa_{c}$ shrinking as $\tau$ grows.

\begin{figure}
\centering \includegraphics[scale=0.25]{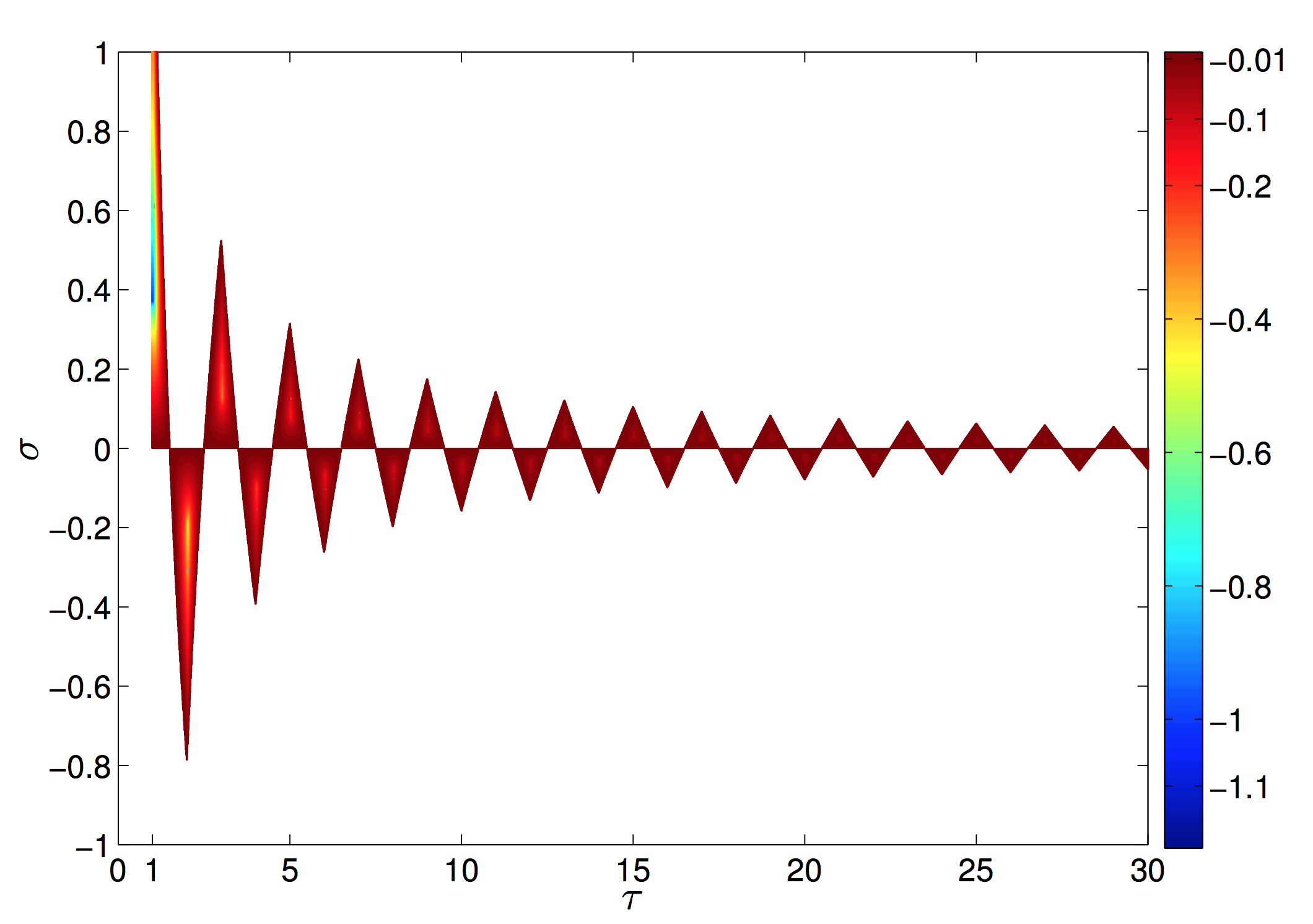}
\caption{Synchronization map in the $\sigma\times\tau$ parameter space. The color
scale represents $\Re(\lambda)<0$ in which $\lambda$ is a solution
of \eqref{eq_transcedental_SL_periodic} with maximal real part. The
white color stands for the instability region ($\Re(\lambda)>0$).
The parameters used were $\alpha=1$ and $\beta=\pi$.}
\label{fig_colormap_maxlamb} 
\end{figure}

The parameter values $-1\leq\sigma\leq1$ of the color map in Fig.~\ref{fig_colormap_maxlamb}
can be related to more complicated connected network for coupling
parameter values in the range $-1/\rho_{L}\leq\kappa<1/\rho_{L}$.
The interchange between stability domains occurs at the values
$\tau=(\pi+2M\pi)/(2\beta)=(2M+1)/2$, $M\in\mathbb{N}$, where we
considered $\beta=\pi$. For such values of time delays no stable
synchronization is attained.

\subsection{Characteristic time}

\label{sec_transient}

In this section, we discuss the transient time such that the trajectories
enter a small vicinity of the synchronization manifold. In particular,
we are interested in the scaling of the transient time with coupling
strength $\kappa$ and time delay $\tau$.

The real parts of the eigenvalues of the linearized system can be
estimated rigorously, so the characteristic time should be related
to the properties of the linearized system. The solutions $\xi(t)$
of Eq. \eqref{eq_xi}, in the case of exponential stability, decay
as 
\begin{equation}
\|\xi(t)\|\le Ce^{-\eta t}\label{eq_xi_decaying}
\end{equation}
with $\eta>0$ to be the largest real part of the eigenvalues. Hence,
we define the \textit{characteristic time} $t_{\text{tr}}$ as 
\[
t_{\text{tr}}=1/\eta.
\]
The characteristic time measures how fast the slowest solution $\xi(t)$
approaches $0$. One can write $\eta=-\gamma_{\max}/\tau$ where $\gamma_{\max}$
is the maximum of the real part of the asymptotic continuous spectrum
given in Eq. \eqref{eq_gamma_assimpt_spectrum}. It can be computed
as 
\begin{align*}
\gamma_{\max}=-\min_{l}\inf_{\omega}\ln\vert g_{l}(\omega)\vert+\max_{j}\ln\kappa\vert\mu_{j}\vert=\ln\frac{\kappa\rho_{L}}{r_{0}}=\ln\frac{\kappa}{\kappa_{c}}.
\end{align*}
Hence, the characteristic time is 
\begin{equation}
t_{\text{tr}}(\kappa)=-\tau\ln^{-1}\frac{\kappa}{\kappa_{c}}.\label{eq_transient}
\end{equation}
Clearly, $t_{\text{tr}}\to\infty$ as $\kappa\to\kappa_{c}.$

Fig. \ref{fig:Transient-time} shows a test for characteristic time
given by Eq. \eqref{eq_transient} using two coupled Stuart-Landau
oscillators with parameters $\alpha=-1$, $\beta=\pi$, and $\tau=20$.
In this example, Eq. \eqref{eq_transient} reads $t_{\text{tr}}(\kappa)=-20\ln^{-1}(2\kappa)$.

\begin{figure}[ht]
\centering \includegraphics[scale=0.4]{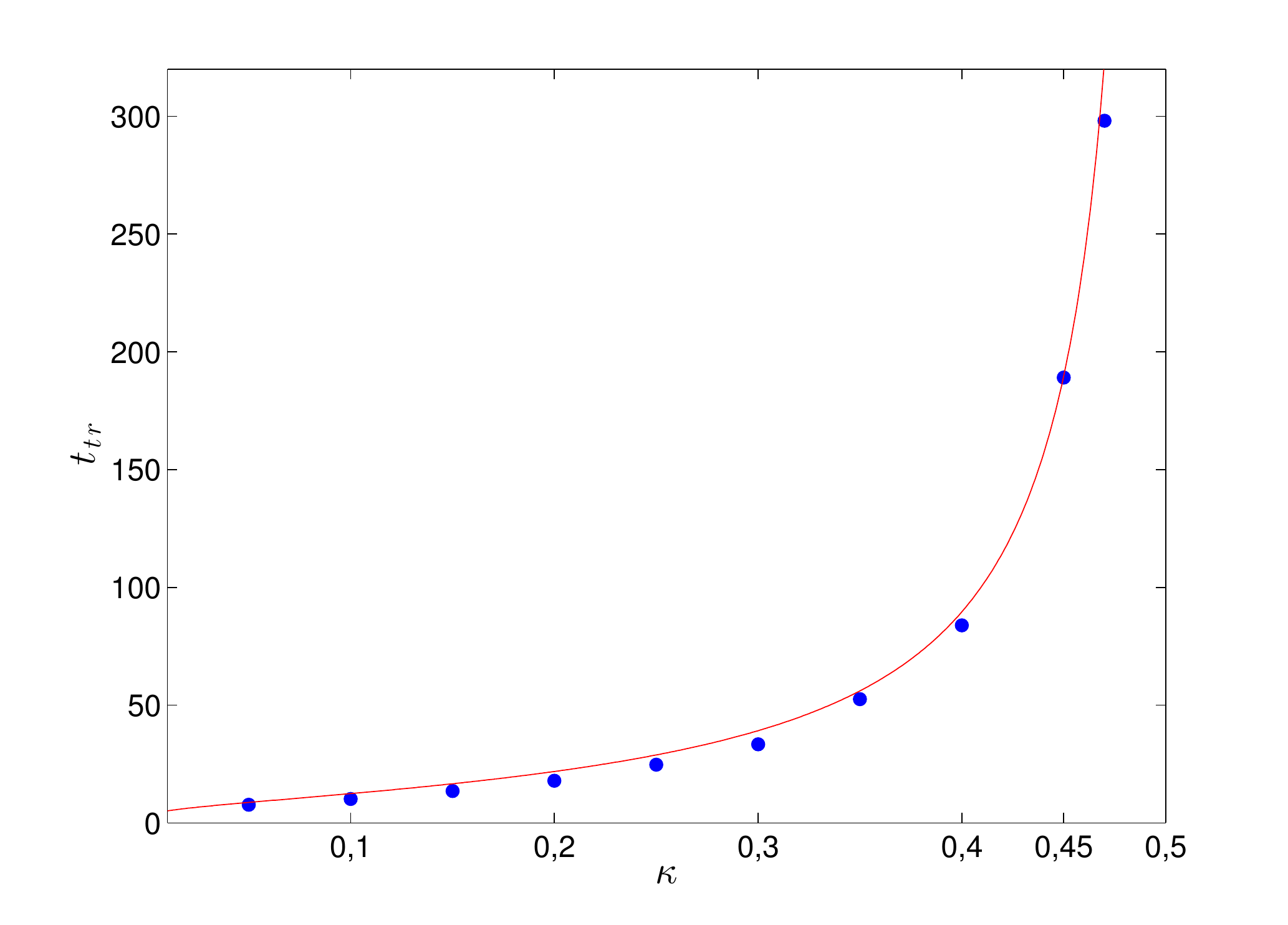}
\caption{Characteristic time for the synchronization of two Stuart-Landau coupled
oscillators. The red curve is $t_{\text{tr}}(\kappa)=20\ln^{-1}(2\kappa)$.
The blue dots were obtained by fixing $\kappa$ and computing $\eta$,
which stands for the angular coefficient of Eq. \eqref{eq_xi_decaying}
in log scale in which $||\xi(t)||=||x_{1}(t)-x_{2}(t)||$, and then
taking $t_{\text{tr}}=1/\eta$. The parameters used were $\alpha=-1$,
$\beta=\pi$ and $\tau=20$. The history functions were taken as constant
and non-zero. \label{fig:Transient-time} }
\end{figure}

\section{Synchronization loss versus network structure}

\label{sec_sync_loss}

Here, we explore the relationship between growing networks (with strongly
delayed connection) and its synchronization window in the case of
equilibrium. The main results are the corollaries \ref{theorem_BA}
and \ref{theorem:ER}.

We will concentrate on the behavior of large networks such as Barab\'asi-Albert
scale-free network, Erd\"os-R\'eniy random network, and some regular graphs,
for which the expressions for the Laplacian spectral radius $\rho_{L}$
are known.

First, let us take a look at how regular graphs respond to the synchronization
(using the network model \eqref{eq:net_mod_lap} with long delay)
in the limit of large network size. Table \ref{tab:laplacian} lists
spectral radius of the Laplacian matrix of the main regular graphs:
complete, ring, star, and path.

\begin{table}[ht!]
\centering{}%
\begin{tabular}{|c|c|c|}
\hline 
Graph  & $\rho_{L}$  & Synchronization window\tabularnewline
\hline 
\hline 
Complete  & $n$  & $(0,r_{0}/n)$\tabularnewline
\hline 
Ring  & $\begin{array}{c}
4\mbox{ if \ensuremath{n} is even}\\
2+2\cos\left({2\pi}/{n}\right)\mbox{ if \ensuremath{n} is odd}
\end{array}$  & $(0,r_{0}/4)$ or $\left(0,r_{0}/(2+2\cos\left({2\pi}/{n}\right))\right)$\tabularnewline
\hline 
Star  & $n$  & $(0,r_{0}/n)$\tabularnewline
\hline 
Path  & $2+2\cos\left({\pi}/{n}\right)$  & $\left(0,r_{0}/\left(2+2\cos\left({\pi}/{n}\right)\right)\right)$\tabularnewline
\hline 
\end{tabular}\caption{Laplacian spectral radius $\rho_{L}$ and synchronization window for
the coupling parameter $\kappa$ (for strong delay) of some regular
graphs. }
\label{tab:laplacian} 
\end{table}

Using Theorem \ref{theo:fp}, and Table \ref{tab:laplacian} we observe
that: 
\begin{itemize}
\item[–] For strong delay and a large network size $n\to\infty$, the synchronization
manifold tends to be always unstable in networks of the types \emph{Complete}
or \emph{Star} provided that the coupling strength $\kappa$ does
not scale with $n$. 
\item[–] For strong delay and a large network size $n\to\infty$, the synchronization
manifold tends to be stable for a certain interval of the coupling
strength $\kappa\in(0,r/4)$ in simple networks of the types \emph{Ring}
or \emph{Path}. 
\end{itemize}
Now, let us consider some complex networks. We relate the synchronization
condition to the graph structure for two important examples of complex
networks.

\emph{Homogeneous networks}, characterized by a small disparity in
the node degrees. A canonical example is the Erd\"os-R\'eniy (ER) random
network: Starting with $n$ nodes the graph is constructed by connecting
nodes randomly. Each edge is included in the graph with probability
$0<p<1$ independent from every other edge. If $p=p_{0}\ln n/n$ with
$p_{0}>1$ then the ER network is almost surely connected \cite{bollobas2001random}.
We will consider that $p$ is chosen so that the ER network is connected.

\emph{Heterogeneous networks}, where some nodes (called hubs) are
highly connected whereas most of the nodes have only a few connections.
A canonical example of such networks is the Barab\'asi-Albert (BA) scale-free
network. One way to construct it is to start with two nodes and a
single edge that connect them. Then at each step, a new node is created
and it is connected with one of the preexisting nodes with a probability
proportional to its degree. This process is called preferential attachment.

More details about the construction, structure, and dynamics of such
networks can be found in \cite{newman2003,Boccaletti2006}. Illustrations
of Erd\"os-R\'eniy (ER) random networks (homogeneous) and Barab\'asi-Albert
(BA) Scale-Free networks (heterogeneous) can be seen in Figure \ref{fig:BA}.

\begin{figure}[!ht]
\centering
\noindent\makebox[1\textwidth]{
\includegraphics[scale=0.35]{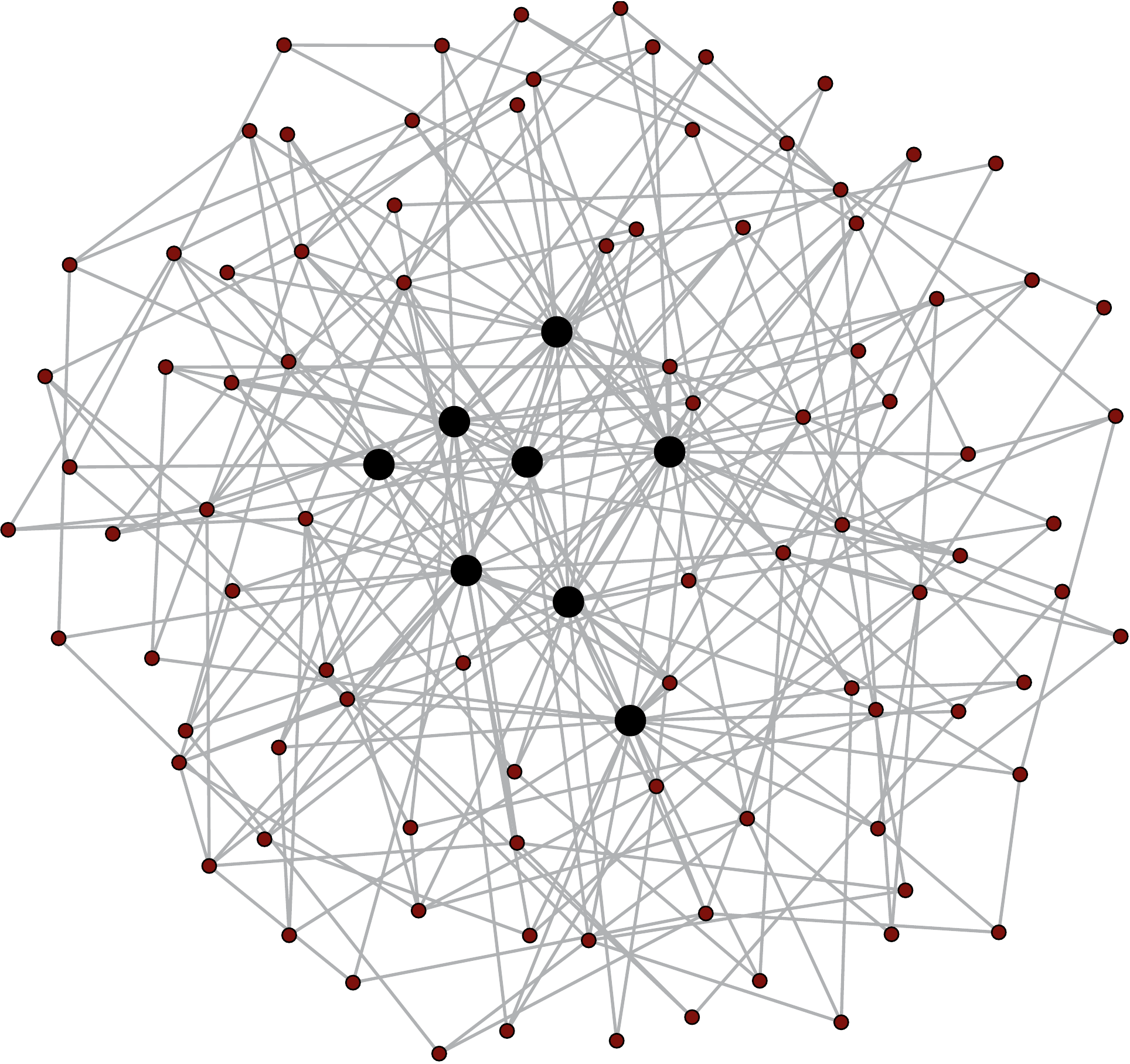} $\qquad$
\includegraphics[scale=0.35]{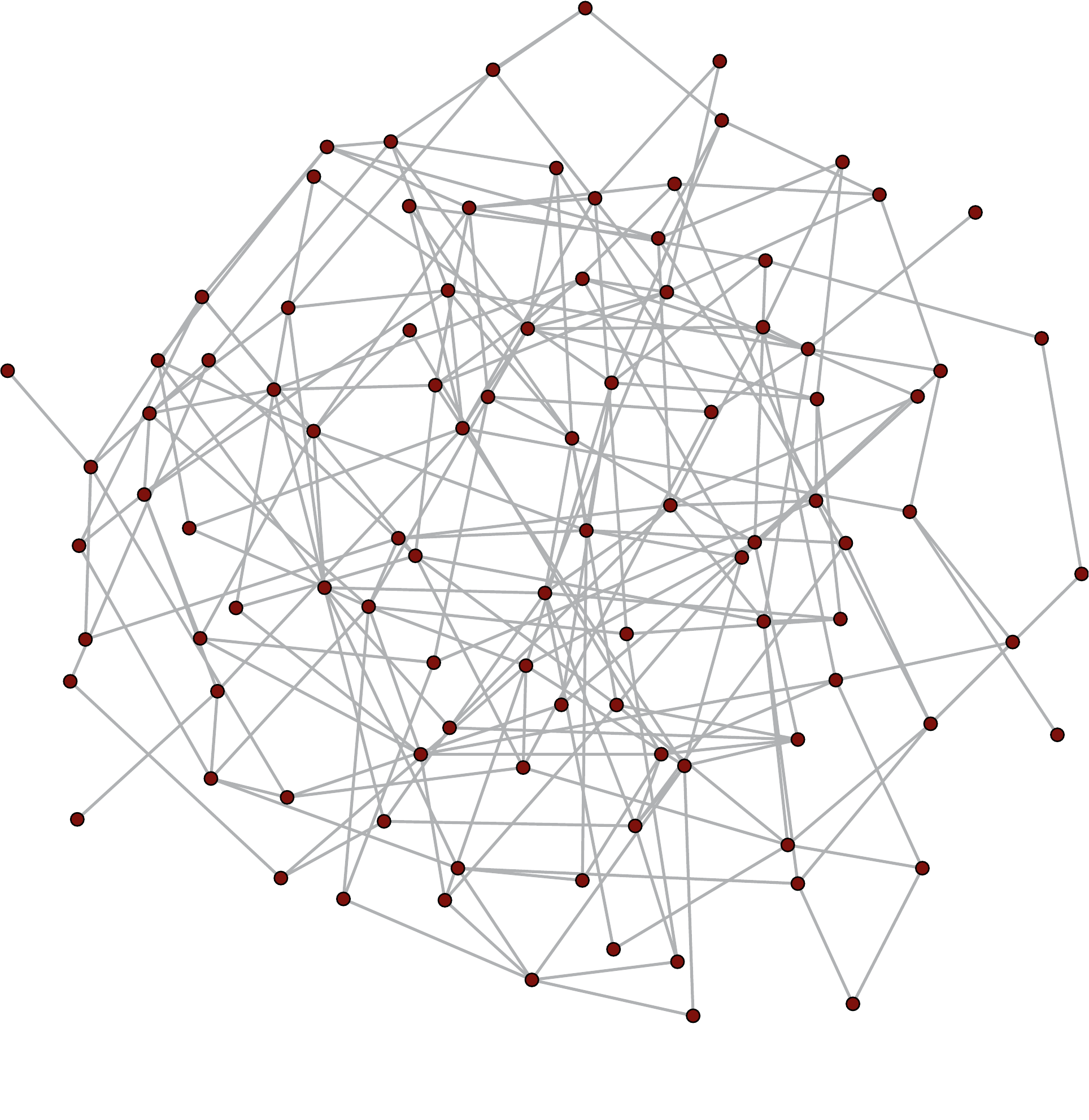}
}
\caption{Illustrations of a BA network (left) and an ER network (right), both
with $n=100$. Some hubs in the BA network are highlighted with black
color and bigger size.}
\label{fig:BA} 
\end{figure}

The response of the BA and ER networks to synchronization under the
considered delayed model already stated in Sec. \ref{sec_main_results}
and encoded in corollaries \ref{theorem_BA} and \ref{theorem:ER}.
The both results says that ER and BA networks tend to not have stable
synchronization in the limit of large $n$. But, ER networks, losses
stable synchronization at a slow rate compared to BA networks.

Here we prove the both cited corollaries. The proof of Corollary \ref{theorem_BA}
is based on the following well-known results. 

\begin{lemma}[Ref. \cite{Mohar2004}] 
\label{lemma:rhoL_and_gmax} 
Let $g_{\max}$
denote the largest degree of a undirected network $G$ of size $n$.
Then the spectral radius $\rho_{L}$ of the Laplacian matrix $L$
of $G$ has the following estimates: 
\begin{equation}
\frac{n}{n-1}g_{\max}\leq\rho_{L}\leq2g_{\max}.\label{eq_bounds_for_rhoL}
\end{equation}
\end{lemma}

\begin{lemma}[Ref. \cite{Mori}] 
\label{lemma_BA} 
Consider a BA
undirected network of size $n$ and $g_{\max}$ its largest degree.
With probability $1$ we have 
\begin{equation}
\lim_{n\to\infty}n^{-1/2}g_{\max}=\mu;\label{eq_aux2_cor2}
\end{equation}
the limit is almost surely positive and finite.
\end{lemma} 

\begin{proof}[Proof of Corollary \ref{theorem_BA}] 
The Lemma
\ref{lemma_BA} implies that the maximum degree of a BA network grows
as $\sqrt{n}$ when we let $n\to\infty$. Therefore, using Eq. \eqref{eq_bounds_for_rhoL}
and \eqref{eq_aux2_cor2} we see that $\rho_{L}\geq\mu\sqrt{n}$ with
probability 1 in the limit of large $n$ which implies that the synchronization
window, that is, the interval $(0,r/\rho_{L})$ (see Theorem \ref{theo:fp})
shrinks and vanishes with large $n$. 
\end{proof} 

\begin{remark}
Corollary \ref{theorem_BA} shows that the synchronization manifold
tends to be unstable in large and strongly delayed BA (heterogeneous)
networks if the coupling strength is not scaled with $n$ or $\tau$.
\end{remark} 
For the proof of Corollary \ref{theorem:ER}, we use Lemma
\ref{lemma:rhoL_and_gmax} and the following result. 

\begin{lemma}[Ref. \cite{Riordan2000}]
\label{lemma:Bollobas_ER} 
Consider an Erd\"os-R\'eniy
network with $n$ nodes, connection probability $0<p<1$ and $q=1-p$.
Then, the probability that the maximum degree $g_{\max}$ being at
most $np+b\sqrt{npq}$, for some constant $b>0$, tends to $1$ as
$n\to\infty$. 
\end{lemma} 

\begin{proof}[Proof of Corollary \ref{theorem:ER}]
For large $n$ and using the probability $p=p_{0}\ln n/n$ with $p_{0}>1$
one get 
\[
g_{\max}\sim p_{0}\ln n+O\left(\sqrt{\ln n}\right).
\]
So, using Lemma \ref{lemma:rhoL_and_gmax}, we end up with 
\begin{align*}
\rho_{L} & \geq\left(\frac{n}{n-1}\right)\left[p_{0}\ln n+O\left(\sqrt{\ln{n}}\right)\right]
\end{align*}
which means that $\rho_{L}\to\infty$ when $n\to\infty$ with a rate
of order $\ln n$. The spectral radius $\rho_{L}$ is the lowest possible
when $p_{0}\to1^{+}$ maintaining the network connected, or, $\rho_{L}$
is the lowest possible when the network crosses the connectivity threshold
becoming connected. 
\end{proof}

Although Corollary \ref{theorem:ER} says that the ER random network
does not allow the synchronization manifold to be stable in the limit
of $n\to\infty$, the rate in which it happens is slow, making it
possible to synchronize very large ER networks with long time delays.

The properties observed for the stability of the synchronization manifold
for BA and ER networks with strong delay, that is, large BA networks
doesn't support strong delay interaction and relatively large ER networks
supports strong delay interaction, are similar to the persistence
of the synchronization when the non-delayed coupling functions are
nonidentical \cite{Maia2016}.

\subsection{Synchronization of BA and ER networks when coupling strength is scaled with $n$}

\label{sec_scaling}

As we have seen from the previous subsection, not only simple but
also complex networks tend to have an always unstable synchronization
manifold with strong delay $\tau$ and large network size $n$.

In many network dynamical models, it is common to normalize the coupling
parameter, in our case $\kappa$, in order to preserve certain behaviors
and properties of the network, such as synchronization, mean field
oscillation, community clustering, etc \cite{Arenas200893,Rodrigues20161,Fortunato201075}.
When the network size is dynamical, for instance, the size of the
network is growing with time, then this normalization becomes even
more important.

With this notion in mind, we can see that in the regime of large network
size, the coupling parameter $\kappa$ should have some natural scaling
depending on the network structure. If those scaling are taken into
account in the network model then the stability of the synchronization
manifold can always be preserved for $n\to\infty$.

For example, if we know beforehand that we have a large simple network
of the type Complete (or Star) then the natural scaling of the coupling
parameter would be 
\[
\kappa\to\frac{\kappa}{n}.
\]

We then state here further consequences of the corollaries \ref{theorem_BA}
and \ref{theorem:ER} when considering the scaling of the coupling
parameter. 
\begin{corollary} 
\label{cor_BA_rescaling} 
Consider a BA scale-free
network with number of nodes equals to $n$ and the network
model \eqref{eq:net_mod_lap} with the coupling parameter 
\[
\kappa=\frac{\kappa_{s}}{\sqrt{n}}.
\]
Then for any large enough size $n$ and long enough delay, there is
always a non-empty interval $I=(0,\kappa_{\max}(n))\subset\mathbb{R}$
such that any synchronous steady state is locally exponentially
stable for all $\kappa_{s}\in I$ and unstable if $\kappa_{s}\in\mathbb{R}\setminus\overline{I}$
(with probability 1). Moreover, the length of this interval converges
to a nonzero constant with $n\to\infty$ with probability 1: 
\[
\lim_{n\to\infty}\kappa_{\max}(n)=\kappa_{\infty},\quad0<\kappa_{\infty}<\infty.
\]

\end{corollary} 

\begin{proof} Any equilibrium 
solution is locally exponentially stable if $0<\kappa<r/\rho_{L}$
and unstable for $\kappa>r/\rho_{L}$ for some $r>0$. For BA scale-free
networks we know that $\rho_{L}\sim\sqrt{n}$ (see Lemma \ref{lemma_BA}).
Then, if we re-scale the coupling parameter as $\kappa\to\kappa_{s}/\sqrt{n}$
the new synchronization condition is $0<\kappa_{s}/\sqrt{n}<r/\rho_{L}\approx r/\sqrt{n}$,
which leads to $0<\kappa_{s}<r$. 
\end{proof} 

\begin{corollary} 
\label{cor_ER}
Consider an ER random network with $n$ nodes and the network model
\eqref{eq:net_mod_lap} with new coupling parameter scaled as 
\[
\kappa=\frac{\kappa_{s}}{\ln n}.
\]
Then for large enough network size $n$ and long enough time delay
$\tau$, there is an interval $I=(0,\kappa_{\max}(n))\subset\mathbb{R}$
such that any equilibrium solution is locally
exponentially stable if $\kappa_{s}\in I$ and unstable if $\kappa_{s}\in\mathbb{R}\setminus\overline{I}$
(with probability 1). Moreover, the length of the interval $I$ converges
to a nonzero constant with $n\to\infty$ with probability 1: 
\[
\lim_{n\to\infty}\kappa_{\max}(n)=\kappa_{\infty},\quad0<\kappa_{\infty}<\infty.
\]
\end{corollary} The proof of the Corollary \ref{cor_ER} follows the same
steps of the Corollary \ref{cor_BA_rescaling} using the Lemma \ref{lemma:Bollobas_ER}.

\section*{Acknowledgments}

The work of D.N.M Maia was supported by CNPq grant
233718/2014-1. E. Macau also acknowledges the support of CNPq and FAPESP2015/50122-0 T. Pereira acknowledges the support of CEPID-CeMeai FAPESP
project 2013/07375-0 and S.Yanchuk acknowledges the support of the German Research
Foundation (DFG) in the framework of the Collaborative Research Center
910. This research is also supported by grant 2015/50122-0   of Sao Paulo Research
Foundation (FAPESP) and DFG-IRTG 1740/2. The authors acknowledge valuable discussions with Jan Philipp Pade and
Stefan Ruschel.

\bibliographystyle{plain}
\bibliography{dp}

\medskip{}
 Received xxxx 20xx; revised xxxx 20xx. \medskip{}

\end{document}